\documentclass[11pt,a4paper]{article}
\usepackage{amsmath}
\usepackage{amssymb}
\usepackage{amsthm}
\usepackage{latexsym}
\usepackage{pstricks}
\usepackage{amsbsy}
\usepackage{amscd}
\usepackage{mathrsfs}
\usepackage{tipa}
\usepackage{txfonts}
\usepackage{amsfonts}
\usepackage{graphicx}
\usepackage{tabularx}
\usepackage{listings}
\usepackage{tikz}
\usepackage{hyperref}
\tikzstyle{dot} = [inner sep=0pt,thick,fill=black,circle,minimum size=2.5pt]
\tikzstyle{line} = [draw, -latex]
\newtheorem{thm}{Theorem}[section]
\newtheorem{cor}[thm]{Corollary}
\newtheorem{lem}[thm]{Lemma}

\newtheorem{exm}{Example}
\newtheorem{prop}[thm]{Proposition}
\newtheorem{defn}[thm]{Definition}
\newtheorem{rem}[thm]{Remark}
\newtheorem{defn-prop}[thm]{Definition-Proposition}

\usepackage[all,poly]{xy}
\usepackage[all]{xy}
\usepackage{xypic}

\begin{document}

\begin{center}
{\Large \bf On rooted cluster morphisms and cluster structures in $2$-Calabi-Yau triangulated categories \footnote{Supported by the NSF of China (Grant 11131001)}}

\bigskip

{\large Wen Chang\footnote{School of Mathematics and Information Science, Shaanxi Normal University, Xi'an 710062, China \&\\	
Department of Mathematical Sciences, Tsinghua University, Beijing 10084, China\\
Email: {\tt changwen161@163.com}} and
 Bin Zhu\footnote{Department of Mathematical Sciences, Tsinghua University, Beijing 10084, China\\
Email: {\tt bzhu@math.tsinghua.edu.cn}}}
\end{center}

\begin{abstract}
We study rooted cluster algebras and rooted cluster morphisms which were introduced in \cite{ADS13} recently and cluster structures in $2$-Calabi-Yau triangulated categories. An example of rooted cluster morphism which is not ideal is given, this clarifies a doubt in \cite{ADS13}. We introduce the notion of freezing of a seed and show that an injective rooted cluster morphism always arises from a freezing and a subseed. Moreover, it is a section if and only if it arises from a subseed. This answers the Problem 7.7 in \cite{ADS13}. We prove that an inducible rooted cluster morphism is ideal if and only if it can be decomposed as a surjective rooted cluster morphism and an injective rooted cluster morphism. For rooted cluster algebras arising from a $2$-Calabi-Yau triangulated category $\mathcal{C}$ with cluster tilting objects, we give an one-to-one correspondence between certain pairs of their rooted cluster subalgebras which we call complete pairs (see Definition \ref{def of complete pairs}) and cotorsion pairs in $\mathcal{C}$.
\end{abstract}

\def\s{\stackrel}
\def\Longrightarrow{{\longrightarrow}}
\def\A{\mathcal{A}}
\def\B{\mathcal{B}}
\def\C{\mathcal{C}}
\def\D{\mathcal{D}}
\def\F{\mathcal{F}}
\def\T{\mathcal{T}}
\def\R{\mathcal{R}}
\def\P{\mathcal{P}}
\def\S{\Sigma}
\def\H{\mathcal{H}}
\def\U{\mathscr{U}}
\def\V{\mathscr{V}}
\def\M{{\bf{Mut}}}
\def\L{\mathcal{L}}
\def\W{\mathscr{W}}
\def\X{\mathscr{X}}
\def\Y{\mathscr{Y}}
\def\x{{\mathbf x}}
\def\ex{{\mathbf{ex}}}
\def\fx{{\mathbf{fx}}}
\def\I{\mathcal {I}}
\def\add{\mbox{add}}
\def\Aut{\mbox{Aut}}
\def\coker{\mbox{coker}}
\def\deg{\mbox{deg}}
\def\diag{\mbox{diag}}
\def\dim{\mbox{dim}}
\def\End{\mbox{End}}
\def\Ext{\mbox{Ext}}
\def\Hom{\mbox{Hom}}
\def\Gr{\mbox{Gr}}
\def\id{\mbox{id}}
\def\Im{\mbox{Im}}
\def\ind{\mbox{ind}}
\def\mod{\mbox{mod}}
\def\mul{\multiput}
\def\c{\circ}
\def \text{\mbox}

\newcommand{\Z}{\mathbb{Z}}
\newcommand{\Q}{\mathbb{Q}}
\newcommand{\N}{\mathbb{N}}

\def\MM#1{{\bf{(CM#1)}}}

\hyphenation{ap-pro-xi-ma-tion}

\medskip



\section{Introduction}
Cluster algebras were introduced by Sergey Fomin and Andrei Zelevinsky in \cite{FZ02} with the aim of giving an algebraic frame for the study of the canonical bases of quantum groups and the total positivity in algebraic groups. The further developments were made in a series of papers \cite{BFZ05,FZ03,FZ07}. It has turned out that these algebras have been linked to various areas of mathematics, for example, Poisson geometry, algebraic geometry, discrete dynamical system, Lie theory and representation theory of finite dimensional algebras.\\

After a number of studies of the cluster algebras in a combinatorial framework, it seems that a construction of a categorical framework is necessary. In \cite{ADS13}, the authors introduced a category ${\bf{Clus}}$ of rooted cluster algebras (see Definition \ref{def of RCA}) of geometry type with non-invertible coefficients. The only difference between a rooted cluster algebra and a cluster algebra is the emphasis of the initial seed in the rooted cluster algebra. The morphisms in this category, which are called the rooted cluster morphisms (see Definition \ref{def of RCM}), are ring homomorphisms which commute with the mutations of rooted cluster algebras. Those bijective morphisms, which are called cluster automorphisms, are also studied in \cite{ASS12,ASS13,CZ15a,CZ15b,CZ15c,N13,S10}. In this paper, we study rooted cluster algebras, rooted cluster morphisms furthermore, and rooted cluster subalgebras. In particular we answer some questions appeared in \cite{ADS13} and give a relation between complete pairs (see Definition \ref{def of complete pairs}) of rooted cluster subalgebras of rooted cluster algebras arising from $2$-Calabi-Yau triangulated categories and cotorsion pairs in these categories. The relation between complete pairs (see Definition \ref{def of complete pairs}) of rooted cluster subalgebras of rooted cluster algebras arising from stably $2$-Calabi-Yau Frobenius categories and cotorsion pairs in these categories will be given in [CZZ]. We introduce the main results and the structure of the paper as follows:\\

We study rooted cluster morphisms and ideal rooted cluster morphisms in Section \ref{Sec RCM}. In subsection \ref{Sec RCMP}, we recall the definitions of rooted cluster algebras and rooted cluster morphisms. In subsection \ref{Sec RCMI}, we consider the ideal rooted cluster morphism, which is a special rooted cluster morphism such that its image coincides with the rooted cluster algebra of its image seed (see Definition-Proposition \ref{def of ideal}(1)).
In general, a rooted cluster morphism is not ideal, we give such an example (Example \ref{Ex of nonideal RCM}), which clarifies a doubt in \cite{ADS13} (compare Problem 2.12 in \cite{ADS13}). Note that Gratz also obtains this result in \cite{G14}. Then we prove in Theorem \ref{ideal morp} that three kinds of important rooted cluster morphisms are ideal. Further, we show that ideal rooted cluster morphisms have some nice properties. It is shown in Theorem \ref{image of subalgebra}(2) that an ideal rooted cluster morphism is a composition of a surjective rooted cluster morphism and an injective rooted cluster morphism. In Theorem \ref{Inverse of rem of ideal} we prove that the inverse statement is also true for an inducible rooted cluster morphism, that is, the rooted cluster morphism which is induced from a ring homomorphism between the corresponding ambient rings (see Definition \ref{def of RCMs}).\\


In subsection \ref{Sec RCMIce}, we introduce the decompositions of an ice valued quiver and of an extended skew-symmetrizable matrix. And then we decompose a seed by decomposing the corresponding exchange matrix at the beginning of subsection \ref{Sec RCMT}. Under these decomposition, the tensor decomposition of a rooted cluster algebra is described in Theorem \ref{inj of indec comp}. Then we introduce the notion of freezing of a seed in Definition \ref{def of frozen} by freezing some initial exchangeable variables.\\

Subsection \ref{Sec RCMIS} is devoted to study the rooted cluster subalgebras of a rooted cluster algebra. Given a seed $\S$ and a freezing $\S_f$ of $\S$, it is proved in Definition-Proposition \ref{inj of frozen} that the rooted cluster algebra $\A(\S_f)$ is a rooted cluster subalgebra of $\A(\S)$, that is, there is an injective rooted cluster morphism from $\A(\S_f)$ to $\A(\S)$. Example \ref{subalg of subseed} shows that some special subseeds also give rooted cluster subalgebras. In fact, we prove in Theorem \ref{Prop of inj} that for a given rooted cluster algebra $\A(\S)$, each of its rooted cluster subalgebra $\A(\S')$ comes from a freezing and a subseed, and the corresponding injection from $\A(\S')$ to $\A(\S)$ is a section if and only if the image seed $f(\S')$ is a subseed of $\S$. The last statement answers the Problem 7.7 in \cite{ADS13}.\\

By using the results in Section \ref{Sec RCM}, the second purpose of the paper is to establish a relation between the rooted cluster subalgebras of rooted cluster algebras arising from $2$-Calabi-Yau triangulated categories with cluster tilting subcategories and the cotorsion pairs in these categories. Recall that cluster categories were firstly introduced in \cite{BMRRT06} (see also \cite{CCS06} for type $A_{n}$) as a categorification of cluster algebras. The stable module categories of the preprojective algebras of Dynkin type were considered for a similar purpose \cite{GLS06,GLS08}. Both of these categories are 2-Calabi-Yau triangulated categories with cluster tilting subcategories. The cluster structure given by the cluster tilting subcategories is defined in \cite{BIRS09} and also be studied in \cite{FK10}. A 2-Calabi-Yau triangulated category with a cluster structure can be viewed as a categorification of cluster algebras associated to the cluster tilting subcategories. There are many works on this topic, see survey papers and references cited there \cite{K12,Re10}. 
A cotorsion pair in a triangulated category was introduced in \cite{IY08}, see also in \cite{KR07}, and studied in \cite{AN12,HJR11,HJR12,Ng10,Na11,ZZ11,ZZ12} and many others furthermore. We recall some basic definitions and results on cotorsion pairs in subsection \ref{Sec TriP}.\\

In subsection \ref{Sec TriSubfactor}, we show that for a functorially finite rigid subcategories $\I$ in a $2$-Calabi-Yau triangulated category $\C$ with a cluster structure, the subfactor triangulated category $^{\bot}\I[1]/\I$ inherits a cluster structure from $\C$. Subsection \ref{Sec TriCot} is devoted to study the cluster substructures in cotorsion pairs. We prove in Theorem \ref{Them in TriC} that in a $2$-Calabi-Yau triangulated category $\C$ with a cluster structure given by its cluster tilting subcategories, if the core $\I$ of a cotorsion pair can be extended as a cluster tilting subcategory, then both the torsion subcategory and the torsionfree subcategory in the cotorsion pair have cluster substructures.\\

By using the cluster map $\varphi$ defined in \cite{FK10}, we prove in Theorem \ref{thm of Last} that there is a bijection between cotorsion pairs with core $\I \subseteq \T$ and complete pairs of rooted cluster subalgebras of $\A(\T)$ with coefficient set $\varphi(\I)$, where $\T$ is a cluster tilting subcategory in $\C$.
Finally, by gluing indecomposable components of $\A(\T_\I)$, we classify the cotorsion pairs with core $\I$ in $\C$. \\

\section{Rooted cluster morphisms}\label{Sec RCM}
\subsection{Preliminaries of rooted cluster morphisms}\label{Sec RCMP}

Cluster algebras were introduced in \cite{FZ02}. For the convenience of studying morphisms between cluster algebras, in \cite{ADS13}, the authors introduced rooted cluster algebras of geometry type with non-invertible coefficients by fixing an initial seed of the cluster algebras. We recall basic definitions and properties on rooted cluster algebras and rooted cluster morphisms in this subsection.
\begin{defn}\cite{FZ02}
\begin{enumerate}
\item A seed is a triple $\Sigma=(\ex,\fx,B)$ where $\ex=\{x_{1},x_{2},\cdot \cdot \cdot ,x_{n}\}$ is a set with n elements and $\x=\ex \sqcup \fx=\{x_{1},x_{2},\cdot \cdot \cdot ,x_{m}\}$ is a set with m elements. Here $m\geqslant n$ are positive integers or countable cardinality. $B=(b_{xy})_{{\x}\times {\ex}}\in M_{{\x}\times {\ex}}(\Z)$ is a locally finite integer matrix with a skew-symmetrizable submatrix $\tilde{B}$ consisting of the first n rows.
\item A seed $\Sigma'=(\ex',\fx',B')$ is called a subseed of a given seed $\Sigma=(\ex,\fx,B)$, if $\ex' \subseteq \ex$, $\fx' \subseteq \fx$ and $B'=B[\ex'\sqcup \fx']$, where $B[\ex'\sqcup \fx']$ is a submatrix of $B$ corresponding to the subset $\ex'\sqcup \fx'$.
\item A seed $\S^{op}=(\ex,\fx,-B)$ is called the opposite seed of a given seed $\Sigma=(\ex,\fx,B)$.
\end{enumerate}
\end{defn}

The set $\x$ is the cluster of $\S$. The elements in $\x$ ($\ex$ and $\fx$ respectively) are the cluster variables (the exchangeable variables and the frozen variables respectively) of $\S$.
The matrix B is called the exchange matrix of $\S$. It is an extended skew-symmetrizable integer matrix with the principal part $\tilde{B}$ skew-symmetrizable. We say a seed $\S=(\ex,\fx,B)$ trivial if the set $\ex$ is empty. The rational function field $\F_{\S}=\Q(x_{1},x_{2},\cdot \cdot \cdot ,x_{m})$ is called the ambient field of $\S$.

\medskip
Given a seed $\S$ and an exchangeable variable $x$ of $\S$, we can produce a new seed by a mutation defined as follows:

\begin{defn}\cite{FZ02}.\label{def of mutation}
The seed $\mu_x(\S)=(\mu_x(\ex),\mu_x(\fx),\mu_x(B))$ obtained by the mutation of $\S$ in the direction $x$ is given by:
\begin{enumerate}
				\item $\mu_x(\ex) = (\ex \setminus \{x\}) \sqcup \{x'\}$ where
				$$xx' = \prod_{\substack{y \in \x~; \\ b_{yx}>0}} y^{b_{yx}} + \prod_{\substack{y \in \x~; \\ b_{yx}<0}} y^{-b_{yx}}.$$
				\item $\mu_x(\fx)=\fx$.
				\item $\mu_x(B)=(b'_{yz})_{{\x}\times {\ex}} \in M_{{\x}\times {\ex}}(\Z)$ is given by
					$$b'_{yz} = \left\{\begin{array}{ll}
						- b_{yz} & \textrm{ if } x=y \textrm{ or } x=z~; \\
						b_{yz} + \frac 12 (|b_{yx}|b_{xz} + b_{yx}|b_{xz}|) & \textrm{ otherwise.}
					\end{array}\right.$$	
			\end{enumerate}
\end{defn}

Now, we recall the definition of rooted cluster algebras as follows:

\begin{defn}(\cite{ADS13},Definition 1.4).\label{def of RCA}
 Let $\S=(\ex,\fx,B)$ be a seed.
\begin{enumerate}
		\item  A sequence $(x_1, \cdots , x_l)$ is called $\S$-admissible if $x_1$ is exchangeable in $\S$ and $x_i$ is exchangeable in $\mu_{x_{i-1}} \circ \cdots \circ \mu_{x_1}(\S)$  for every $2 \leqslant i \leqslant l$. Denote by ${\bf{Mut}}(\S) = \{\mu_{x_{n}} \circ \cdots \circ \mu_{x_1}(\S) \ | \ n\geq 0 \text{ and } (x_1, \cdots , x_n) \textrm{ is $\S$-admissible}\}$ the mutation class of $\S$.

		\item A pair $(\S,\A)$ is called a rooted cluster algebra with initial seed $\S$, where $\mathcal A$ is the ${\Z}$-subalgebra of $\F_{\S}$ given by~:
			$$\A={\Z}\left[x \ | \ x \in \bigcup_{(\ex',\fx',B') \in {\bf{Mut}}(\S)} \x' \right].$$
		
	\end{enumerate}
\end{defn}
We call the clusters of seeds in ${\bf{Mut}}(\S)$ the clusters of $(\S,\A)$. The cluster variables (the exchangeable variables and the frozen variables respectively) in these clusters are called cluster variables (the exchangeable variables and the frozen variables respectively) of $(\S, \A)$. In particular, the cluster (cluster variable respectively) of $\S$ is called the initial cluster (cluster variable respectively) of $(\S,\A)$. The elements in the multiplicative group freely generated by frozen variables of $(\S, \A)$ are called the coefficients of $(\S, \A)$. We denote by $\X_\S$ the set of cluster variables in $(\S,\A)$. We always write $(\S, \A)$ as $\A(\S)$ for simplicity.
Note that for a trivial seed $\S=(\emptyset,\fx,B)$, the associated rooted cluster algebra is just the polynomial ring $\Z[\fx]$. We refer the readers to \cite{ADS13} for more examples of rooted cluster algebras. Cluster algebras have many remarkable properties, for example, the Laurent phenomenon. In fact, because the only difference between rooted cluster algebras and cluster algebras is the emphasis of the initial seeds in rooted cluster algebras, these properties are also true for rooted cluster algebras.

\begin{thm}(\cite{FZ02},Theorem 3.1)(\cite{FZ03},Proposition 11.2).\label{Laurent}
  Given a seed $\S=(\ex,\fx,B)$ with $\x=\ex\sqcup \fx=\{x_1,\cdots ,x_n\}\sqcup \{x_{n+1},\cdots ,x_m\}$. Let $\L_{\S,\Z}:=\Z [x_1^{\pm 1},\cdots ,x_n^{\pm 1},x_{n+1},\cdots ,x_m]$ be the localization of polynomial ring $\Z[x_{1},\cdots ,x_m]$ at $x_{1},\cdots ,x_n$. Then

$$\X_{\S} \subseteq
{\bigcap_{\S'\in \M (\S)} \L_{\S',\Z}}~~~~
and~~~~
\A(\S) \subseteq
{\bigcap_{\S'\in \M (\S)} \L_{\S',\Z}}.$$
\end{thm}

\begin{defn}\label{def of RCM}.
(\cite{ADS13},Definition 2.2). Let $\S=(\ex,\fx,B)$ and $\S'=(\ex',\fx',B')$ be two seeds. Denote by $\x=\ex \sqcup \fx$ and $\x'=\ex' \sqcup \fx'$. A ring homomorphism $f: \A(\S)\longrightarrow \A(\S')$ is called a rooted cluster morphism if

\begin{tabularx}{15cm}{lX}
	\MM 1 & $f(\ex) \subset \ex' \sqcup \Z$~; \\
	\MM 2 & $f(\fx) \subset \x' \sqcup \Z$~;\\
	\MM 3 & For every $(f,\Sigma,\Sigma')$-biadmissible sequence $(x_1, \cdots , x_l)$, we have $\mu_{x_l} \circ \cdots \circ \mu_{x_1,\Sigma}(y) = \mu_{f(x_l)} \circ \cdots \circ \mu_{f(x_1),\Sigma'}(f(y))$ for any $y$ in $\x$, where $\mu_{f(x_l)} \circ \cdots \circ \mu_{f(x_1),\Sigma'}(f(y))=f(y)$ when $f(y)$ is an integer. Here a $(f,\Sigma,\Sigma')$-biadmissible sequence $(x_1, \cdots , x_l)$ is a $\S$-admissible sequence such that $(f(x_1),\cdots ,f(x_l))$ is $\S'$-admissible.
\end{tabularx}

\end{defn}

After introduced morphisms between rooted cluster algebras in \cite{ADS13}, the authors proved that these rooted cluster morphisms consist the set of morphisms of a category $\bf{Clus}$ with the objects given by rooted cluster algebras. This category has countable coproducts but no products. And in this category, the monomorphisms coincide with the injections but not all the epimorphisms are the surjections. They also introduced the following ideal rooted cluster morphisms to get better understanding of rooted cluster morphisms.

\begin{defn-prop}\label{def of ideal}
Given two seeds $\Sigma=(\ex,\fx,B)$ and $\Sigma'=(\ex',\fx',B')$ with  $\x=\ex \sqcup \fx$ and $\x'=\ex' \sqcup \fx'$. Let $f: \A(\S)\rightarrow \A(\S')$ be a rooted cluster morphism.
\begin{enumerate}
\item (\cite{ADS13},Definition2.8). The image seed of $\S$ under $f$ is $f(\S)=(\ex'\cap f(\ex),{\x'\cap f(\x)}\setminus {\ex'\cap f(\ex)},B'[\x'\cap f(\x)])$.
\item (\cite{ADS13}, Lemma 2.10, Definition 2.11). We have $\A(f(\S))\subseteq f(\A(\S))$. If the equality holds, then $f$ is called an ideal rooted cluster morphism.
\end{enumerate}
\end{defn-prop}

Now we define rooted cluster subalgebras, which is a modified version and a slight generalization of subcluster algebras defined in \cite{BIRS09} (see Remark \ref{rem of inj RCM} for concrete relation of these two definitions). We will study rooted cluster subalgebras in subsection \ref{Sec RCMIS} and  rooted cluster subalgebras arising from a cotorsion pair in a $2$-Calabi-Yau triangulated category in section \ref{Sec Tri}.

\begin{defn}\label{def of injRCM}
If there is an injective rooted cluster morphism $f$ from $\A(\S)$ to $\A(\S')$, then we call $\A(\S)$ a rooted cluster subalgebra of $\A(\S')$.
\end{defn}

\begin{exm}
Given two seeds
			$$\S = \left((x_1,x_2),(x_3),
				\left[\begin{array}{rr}
					0 & 1 \\
					-1 & 0 \\
					0 & -1
				\end{array}\right]
			\right)~and~
\S' = \left((x_1,x_2,x_3),\emptyset,
				\left[\begin{array}{rrr}
					0 & 1 & 0\\
					-1 & 0 & 1\\
					0 & -1 & 0
				\end{array}\right]
			\right).$$
It is easy to check that the identity $f$ on $\F_{\S}=\F_{\S'}=\Q(x_{1},x_{2},x_{3})$ induces a rooted cluster morphism $f'$ from  $\A(\S)$ to $\A(\S')$. Since $f$ is injective, $f'$ is also injective. Therefore $\A(\S)$ is a rooted cluster subalgebra of $\A(\S')$.
\end{exm}

\begin{exm}\label{subalg of subseed}
Given a seed $\Sigma=(\ex,\fx,B)$, let $\Sigma'=(\ex',\fx',B')$ be a subseed of $\S$. Then the natural injection from $\ex' \cup \fx'$ to $\ex \cup \fx$ induces an injective ring morphism $j$ from $\L_{\S',\Z}$ to $\L_{\S,\Z}$. If $j$ induces a rooted cluster morphism $f$ from
$\A(\S')$ to $\A(\S)$, then $f$ is injective and $\A(\S')$ is a rooted cluster subalgebra of $\A(\S)$. We will show in Theorem \ref{Prop of inj}(4) when the injection $j$ induces a rooted cluster morphism $f$.
\end{exm}

\subsection{Ideal rooted cluster morphisms}\label{Sec RCMI}
This subsection is devoted to ideal rooted cluster morphisms. For a seed $\S=(\ex,\fx,B)$ with $\x=\ex\sqcup \fx=\{x_1,\cdots ,x_n\}\sqcup \{x_{n+1},\cdots ,x_m\}$, we introduce $\L_{\S,\Q}:=\Q[x_1^{\pm1},\cdots ,x_n^{\pm1},x_{n+1},\cdots ,x_m]$ as the ambient ring of the rooted cluster algebra $\A(\S)$.

\begin{defn}\label{def of RCMs}
Given $\Sigma=(\ex,\fx,B)$ and $\Sigma'=(\ex',\fx',B')$ be two seeds with  $\x=\ex \sqcup \fx$ and $\x'=\ex' \sqcup \fx'$. Let $f: \A(\S)\rightarrow \A(\S')$ be a rooted cluster morphism.
\begin{enumerate}
\item We call $f$ inducible if it can be lifted as a ring homomorphism between the corresponding ambient rings.
\item We call $f$ explicit if it is uniquely determined by the images of the initial cluster variables. More precisely, if there is a rooted cluster morphism $f'$ from $\A(\S)$ to $\A(\S')$ coinciding with $f$ on the initial cluster variables, then $f=f'$.
\end{enumerate}
\end{defn}

It is easy to see that a rooted cluster morphism is inducible if and only if the image of any initial exchangeable cluster variable is not zero. By the definitions, an inducible rooted cluster morphism is explicit. Conversely, we state the following problem.\\

{\bf{Problem~1:}} Whether any explicit rooted cluster morphism is inducible or not?\\

There exists rooted cluster morphism which is not explicit, see the following example.

\begin{exm}\label{ex of RCM}
Consider the seed
			$$\S = \left((x_1,x_2),(x_3,x_4),
				\left[\begin{array}{rr}
					0 & 1 \\
					-1 & 0 \\
					0 & -1 \\
                    0 & 0
				\end{array}\right]
			\right)$$
and the rooted cluster algebra $\A(\S) = \Z[x_1,x_2,x_3,x_4,\frac{1+x_2}{x_1},\frac{x_1+x_3}{x_2},\frac{x_1+x_3+x_2 x_3}{x_1 x_2}]$.
Define a map
			$$f: \left\{\begin{array}{rcl}
				x_1 & \mapsto & 0 \\
				x_2 & \mapsto & -1 \\
				x_3 & \mapsto & 0 \\
                x_4 & \mapsto & x_1 \\
				\frac{1+x_2}{x_1} & \mapsto & y \\
				\frac{x_1+x_3}{x_2} & \mapsto & 0 \\
                \frac{x_1+x_3+x_2 x_3}{x_1 x_2} & \mapsto & -1
			\end{array}\right.$$
where $y$ is any given element in $\A(\S)$.
One can easily check that this map induces a ring homomorphism from $\A(\S)$ to itself.
Moreover, there is no $(f,\S,\S)$-biadmissible sequence, thus $f$ is a rooted cluster morphism. Because the image of the cluster variable $\frac{1+x_2}{x_1}$ can be chosen as any element in $\A(\S)$, this rooted cluster morphism is not explicit.
\end{exm}

\medskip

The following example shows that not all the rooted cluster morphisms are ideal. This clarifies a double in \cite{ADS13}.

\begin{exm}\label{Ex of nonideal RCM}
Consider the rooted cluster morphism $f$ in \text{Example}~\ref{ex of RCM}. Let $y=x_2$, then we have $f(\mathcal A(\Sigma)) = \Z[x_1,x_2] \text{ and }$
			$$f(\Sigma) = \left(\varnothing,(x_1),
				\left[0\right]\right).$$
			Therefore
			$$\mathcal A(f(\Sigma)) = \Z[x_1]$$
			so that $\mathcal A(f(\Sigma)) \subsetneqq f(\mathcal A(\Sigma))$ and thus $f$ is not ideal.
\end{exm}

We notice that this counterexample is not an explicit rooted cluster morphism and therefore the degree of freedom in choosing the images of cluster variables is increased significantly. Thus we state the following problem which can be viewed as an improved version of Problem 2.12 in \cite{ADS13}.\\

{\bf{Problem~2:}}
Whether every explicit rooted cluster morphism is ideal or not?\\

In \cite{ADS13}, the authors raised a problem that characterize rooted cluster morphisms which are ideal. We have the following answers in several important cases.

\begin{thm}\label{ideal morp}
 Given two seeds $\S=(\ex,\fx,B)$ and $\S'=(\ex',\fx',B')$ with  $\x=\ex\sqcup \fx$ and $\x'=\ex' \sqcup \fx'$. Let $f: \A(\S)\longrightarrow\A(\S')$ be a rooted cluster morphism. Then $f$ is ideal if one of the following conditions is satisfied:

     \begin{enumerate}
     \item $f(\ex) \subset \ex'$;
     \item $f$ is inducible and $\S$ is finite acyclic;
     \item $f$ is inducible and surjective.\\

    \end{enumerate}
\end{thm}

\begin{proof}
\begin{enumerate}
\item From the Definition-Proposition\ref{def of ideal}(2), it is sufficient to show that $f$ maps the cluster variables of $\A(\S)$ to $\A(f(\S))$. Because
     $f(\ex) \subset \ex'$, each $\S$-admissible sequence is $(f,\S,\S')$-biadmissible. Thus for each cluster variable $y=\mu_{x_l} \circ \cdots \circ \mu_{x_2}\circ \mu_{{x_1},\S}(x)$ in $\A(\S)$, where the sequence $(x_1,\cdots ,x_l)$ is $\S$-admissible and $x\in \x$, we have $f(y)=\mu_{f(x_l)} \circ \cdots \circ \mu_{f(x_2)}\circ \mu_{f(x_1),\S'}(f(x))$ because of $\MM 3$. It is clear that $(x_1,\cdots ,x_l)$ can also be viewed as a $(f,\S,f(\S))$-biadmissible sequence. Now we prove that $$\mu_{f(x_l)} \circ \cdots \circ \mu_{f(x_2)}\circ \mu_{f(x_1),\S'}(f(x))=\mu_{f(x_l)} \circ \cdots \circ \mu_{f(x_2)}\circ \mu_{f(x_1),f(\S)}(f(x))$$ and then $f(y)$ belongs to $\A(f(\S))$. Inductively, it is only need to show the case of $l=1$.\\

If $x \neq x_1$, then $\mu_{f({x_1}),\S'}(f(x))=f(\mu_{x_1,\S}(x))=f(x)$. Therefore $f(x) \neq f(x_1)$ and $\mu_{f({x_1}),f(\S)}(f(x))=f(x)=\mu_{f({x_1}),\S'}(f(x))$. In fact, if $f(x)=f(x_1)$, then $\mu_{f({x_1}),\S'}(f(x))=\frac{m_1(\x')+m_2(\x')}{f({x_1})}$, where $m_1(\x')$ and $m_2(\x')$ are monomials of $\x' \setminus f(x_1)$, and not equal to $f(x)$ since the algebraic independence of the initial cluster variables, thus a contradiction.\\

If $x=x_1$, since $f$ is inducible, we have
\begin{eqnarray*}
&f(\mu_{{x_1},\S}({x_1}))
&=f\left(\frac{1}{x_1}\left(\prod_{\substack{y \in \x~; \\ b_{y{x_1}}>0}} f(y)^{b_{y{x_1}}} + \prod_{\substack{y \in \x~; \\ b_{y{x_1}}<0}} f(y)^{-b_{y{x_1}}}\right)\right)\\
&&=\frac{1}{f({x_1})}\left(\prod_{\substack{y \in \x~; \\ b_{y{x_1}}>0}} f(y)^{b_{y{x_1}}} + \prod_{\substack{y \in \x~; \\ b_{y{x_1}}<0}} f(y)^{-b_{y{x_1}}}\right).
\end{eqnarray*}

On the other hand,
$$\mu_{f({x_1}),\S'}(f({x_1}))=\frac{1}{f({x_1})}\left(\prod_{\substack{z \in \x'~; \\ b'_{zf({x_1})}>0}} z^{b'_{zf({x_1})}} + \prod_{\substack{z \in \x'~; \\ b'_{zf({x_1})}<0}} z^{-b'_{zf({x_1})}}\right)~~~and~~~$$
$$\mu_{f({x_1}),f(\S)}(f({x_1}))=\frac{1}{f({x_1})}\left(\prod_{\substack{y \in \x, f(y) \in \x'~; \\ b'_{f(y)f({x_1})}>0}} f(y)^{b'_{f(y)f({x_1})}} + \prod_{\substack{y \in \x, f(y) \in \x'~; \\ b'_{f(y)f({x_1})}<0}} f(y)^{-b'_{f(y)f({x_1})}}\right).$$

Then from $f(\mu_{{x_1},\S}({x_1}))=\mu_{f({x_1}),\S'}(f({x_1}))$ and the algebraic independence of the initial cluster variables, we have
$${\prod_{\substack{y \in \x~; \\ b_{y{x_1}}>0}} f(y)^{b_{y{x_1}}}=\prod_{\substack{z \in \x'~; \\ b'_{zf({x_1})}>0}} z^{b'_{zf({x_1})}}}~~~, ~~~ \prod_{\substack{y \in \x~; \\ b_{y{x_1}}<0}} f(y)^{-b_{y{x_1}}}=\prod_{\substack{z \in \x'~; \\ b'_{zf({x_1})}<0}} z^{-b'_{zf({x_1})}}~~~~~~or$$
$${\prod_{\substack{y \in \x~; \\ b_{y{x_1}}>0}} f(y)^{b_{y{x_1}}}=\prod_{\substack{z \in \x'~; \\ b'_{zf({x_1})}<0}} z^{-b'_{zf({x_1})}}~~~, ~~~ {\prod_{\substack{y \in \x~; \\ b_{y{x_1}}<0}} f(y)^{-b_{y{x_1}}}=\prod_{\substack{z \in \x'~; \\ b'_{zf({x_1})}>0}} z^{b'_{zf({x_1})}}}}~~~.$$
No matter for which case, the above equalities show that the factors of the right hand monomials in these equalities are of the form $f(y)$ with $y \in \x$.
Thus we have equalities
$${\prod_{\substack{z \in \x'~; \\ b'_{zf({x_1})}>0}} z^{b'_{zf({x_1})}}=\prod_{\substack{y \in \x, f(y) \in \x'~; \\ b'_{f(y)f({x_1})}>0}} f(y)^{b'_{f(y)f({x_1})}}}~~~, ~~~ \prod_{\substack{z \in \x'~; \\ b'_{zf({x_1})}<0}} z^{-b'_{zf({x_1})}}=\prod_{\substack{y \in \x, f(y) \in \x'~; \\ b'_{f(y)f({x_1})}<0}} f(y)^{-b'_{f(y)f({x_1})}}~~~~~~or$$
$${\prod_{\substack{z \in \x'~; \\ b'_{zf({x_1})}<0}} z^{-b'_{zf({x_1})}}=\prod_{\substack{y \in \x, f(y) \in \x'~; \\ b'_{f(y)f({x_1})}>0}} f(y)^{b'_{f(y)f({x_1})}}~~~, ~~~ \prod_{\substack{z \in \x'~; \\ b'_{zf({x_1})}>0}} z^{b'_{zf({x_1})}}=\prod_{\substack{y \in \x, f(y) \in \x'~; \\ b'_{f(y)f({x_1})}<0}} f(y)^{-b'_{f(y)f({x_1})}}}$$
respectively. Thus $\mu_{f({x_1}),\S'}(f({x_1}))=\mu_{f({x_1}),f(\S)}(f({x_1}))$. Therefore for each $x\in \x$ we have
\begin{equation}
\mu_{f({x_1}),\S'}(f(x))=\mu_{f({x_1}),f(\S)}(f(x)).
\end{equation}

\item Because $\S$ is finite acyclic, it follows from Theorem1.20 in \cite{BFZ05} that $\A(\S)$ is finite generated and equal to $\Z[x_1,\cdots ,x_m,\mu_{x_1}(x_1),\cdots ,\mu_{x_n}(x_n)]$, where $\ex=\{x_{1},x_{2},\cdot \cdot \cdot ,x_{n}\}$ and $\x=\{x_{1},x_{2},\cdot \cdot \cdot ,x_{m}\}$. Thus to prove that $f(\A(\S))\subseteq \A(f(\S))$, it is only need to show that $f(\mu_{x_i}(x_i))\in \A(f(\S))$ for any $x_i\in \ex$. In fact, if $f(x_i)\in \ex'$, then $f(\mu_{x_i}(x_i))=\mu_{f({x_i}),\S'}(f(x_i))=\mu_{f({x_i}),f(\S)}(f(x_i)) \in \A(f(\S))$ from the equality (1). If $f(x_i)\in \Z $, then we have
$$f(\mu_{x_i}(x_i))=\frac{1}{f({x_i})}\left(\prod_{\substack{x \in \x~; \\ b_{x{x_i}}>0}} f(x)^{b_{x{x_i}}} + \prod_{\substack{x \in \x~; \\ b_{x{x_i}}<0}} f(x)^{-b_{x{x_i}}}\right),$$
since $f$ is inducible. Therefore $f(\mu_{x_i}(x_i))$ belongs to $\A(f(\S))$.

\item It follows from Lemma 3.1 \cite{ADS13} that $f(\S)=\S'$. Thus we have $f(\A({\S}))\subseteq\A({\S'})=\A(f(\S))$ and $f$ is ideal.
\end{enumerate}
\end{proof}

\begin{rem}
In fact, by checking the proof of Lemma 3.1 \cite{ADS13}, one can easily notice that it missed a condition that $f$ should be inducible.
\end{rem}




\subsection{Ice valued quivers}\label{Sec RCMIce}

In this subsection, we recall ice valued quivers and their relations with extended skew- symmetrizable matrices (compare \cite{K12}). Although these contents are standard, we write them down here for the convenience of the readers. Then we decompose ice valued quivers and extended skew- symmetrizable matrices, which will be used to decompose rooted cluster algebras and rooted cluster morphisms in the next subsection.

\begin{defn}\label{def of vq}
A valued quiver is a triple $(Q, v, d)$ where
	\begin{enumerate}
     \item $Q=(Q_0,Q_1)$ is a locally finite, simple-laced quiver without loops nor $2-$cycles. Denote by $Q_0=\{1,2,\cdots ,n\}$ with $n$ the cardinality of $Q_0$ which maybe a countable number;
     \item $v:Q_1\longrightarrow \N^2$ is a function, which maps each arrow $\alpha$ to a non-negative number pair $(v(\alpha)_1,v(\alpha)_2)$;
     \item $d: Q_0\longrightarrow \Z^*$ is a function such that for each vertex $i$ and each arrow $\alpha:i\longrightarrow j$, we have $d(i)v(\alpha)_1={v(\alpha)_2}d(j)$.
    \end{enumerate}
\end{defn}

A valued quiver $(Q, v, d)$ with $v(\alpha)_1=v(\alpha)_2$ for each arrow $\alpha$ in $Q_1$ is called an equally valued quiver. An equally valued quiver $(Q, v, d)$ corresponds to an ordinary quiver $Q'$ in the following way: $Q'$ has the vertex set $Q_0$, and there are $v(\alpha)_1=v(\alpha)_2$ number of arrows from $i$ to $j$ for any vertices $i$,$j \in Q_0$ and $\alpha: i\rightarrow j$. It is easy to check that the above correspondence gives a bijection between the set of equally valued quivers and the set of ordinary quivers without loops nor 2-cycles. Given a valued quiver
$(Q, v, d)$ with $Q_0=\{1,2,\cdots ,n\}$, we associate a matrix $B=(b_{ij})_{n\times n}\in M_{n\times n}(\Z)$ as follows:

$$b_{ij} = \left\{\begin{array}{lll}
                         0& \textrm{ if there is no arrow between}~i \textrm{ and}~j~; \\
						v(\alpha)_1 & \textrm{ if there is an arrow}~ \alpha: i\longrightarrow j~; \\
						-v(\alpha)_2 & \textrm{ if there is an arrow}~ \alpha: j\longrightarrow i~.
					\end{array}\right.$$	

Let $D=(d_{ij})_{n \times n}\in M_{n\times n}(\N)$ be the diagonal matrix with $d_{ii}=d(i)$, $i \in Q_0$, then by Definition \ref{def of vq}(3), $B$ is a skew-symmetrizable matrix in the sense of $DB$ skew-symmetric. Conversely, one can construct a valued quiver from a skew-symmetrizable matrix, we refer to \cite{K12} for more details, and it is easy to check that these procedures give a bijection between the valued quivers with vertex set $\{1,2,\cdots ,n\}$ and the skew-symmetrizable integer matrices with rank $n$.

\begin{defn}
An ice valued quiver $Q=(Q_0,Q_1)$ is a quiver without loops nor 2-cycles, where
	\begin{enumerate}
     \item $Q_0=Q^e_0\sqcup Q^f_0=\{1,2,\cdots ,n\}\sqcup \{n+1,\cdots ,m\}$ with $m\geqslant n$ be positive or countable;
     \item the full subquiver ${Q}^e=(Q^e_0,Q^e_1$) of $Q$ is a valued quiver;
     \item there are no arrows between vertices in $Q^f_0$.
    \end{enumerate}

\end{defn}
The subquiver ${Q}^e$ is called the principal part of $Q$. The vertices in $Q^e_0$ are called the exchangeable vertices while the vertices in $Q^f_0$ are called the frozen vertices.
For the ice valued quiver $Q=(Q_0,Q_1)$, the associated matrix $B=(b_{ij})_{m\times n}\in M_{m\times n}(\Z)$ consists of the following two parts. The principal part of the first $n$ rows is given by the skew-symmetrizable matrix associated with the valued quiver $Q^e$. Given two vertices $i\in Q_0^f$ and $j\in Q_0^e $, the elements in the last $m-n$ rows are as follows:

$$b_{ij} = \left\{\begin{array}{ll}
						n_{ij} &  \textrm{ if there are $n_{ij}$ arrows from $i$ to $j$}~; \\
						-n_{ij}&  ~ \textrm{if there are $n_{ij}$ arrows from $j$ to $i$}~.
					\end{array}\right.$$	

It is easy to see that the above procedure is reversible and gives a bijection between the ice valued quivers with vertex set $Q_0=Q^e_0\sqcup Q^f_0=\{1,2,\cdots ,n\}\sqcup \{n+1,\cdots ,m\}$ and the $m\times n$ extended skew-symmetrizable integer matrices. Using this correspondence, we define the mutation of an ice valued quiver by the matrix mutation defined in Definition \ref{def of mutation}(3). Then the mutation of an ice valued quiver at an exchangeable vertex produces a new ice valued quiver which has the same vertices as the original quiver has. We now introduce the following

\begin{defn}\label{def of indec ivq}
An ice valued quiver $Q$ is called indecomposable, if it is connected and the principal part ${Q}^e$ is also connected.
\end{defn}
It is not hard to see that an ice valued quiver is indecomposable if and only if the corresponding extended skew-symmatrizable matrix is indecomposable in the following sense.
\begin{defn}\label{def of indec ivq}
An extended skew-symmatrizable matrix $B=(b_{ij})_{m\times n}$ is called indecomposable, if it satisfies the following conditions,
	\begin{enumerate}
     \item for any $1 \leqslant s \neq t \leqslant m$, there exists a sequence $(i_0=s, i_1, \cdot \cdot \cdot , i_l=t)$ with $1 \leqslant i_0, i_1, \cdot \cdot \cdot , i_l \leqslant m$ such that $b_{i_ji_{j+1}}\neq 0$ for each $0 \leqslant j \leqslant l-1$;
     \item for any $1 \leqslant s \neq t \leqslant n$, there exists a sequence $(i_0=s, i_1, \cdot \cdot \cdot , i_l=t)$ with $1 \leqslant i_0, i_1, \cdot \cdot \cdot , i_l \leqslant n$ such that $b_{i_ji_{j+1}}\neq 0$ for any $0 \leqslant j \leqslant l-1$.
    \end{enumerate}
\end{defn}

\begin{defn-prop}\label{prop of decop of IVQ}
We decompose an ice valued quiver $Q$ as a collection $\{Q_1, Q_2, \cdots Q_t\}$ of indecomposable ice valued quivers in a unique way. These quivers are called indecomposable components of $Q$.
\end{defn-prop}
\begin{proof}
Given an ice valued quiver $Q$, we denote by ${Q}^e_{1},{Q}^e_{2},\cdots ,{Q}^e_{t}$ all the connected components of ${Q}^e$, where $t$ maybe a countable number. For each component ${Q}^e_{i}$, $1\leqslant i \leqslant t,$ we associate an ice valued quiver $Q_i$ as a full subquiver of $Q$, the exchangeable vertices are the vertices of ${Q}^e_{i}$ while the frozen vertices are those frozen vertices of $Q$ which connected to some vertices in ${Q}^e_{i}$ directly. It is clear that $Q_i$ is indecomposable, and the decomposition is unique.
\end{proof}

We also have the following matrix version of Definition-Proposition \ref{prop of decop of IVQ}.

\begin{defn-prop}\label{prop of decop of Matrix}
We decompose an extended skew-symmatrizable matrix $B$ as a collection $\{B_1, B_2, \cdots B_t\}$ of indecomposable extended skew-symmatrizable matrices  in a unique way. These matrices are called indecomposable components of $B$.
\end{defn-prop}

\begin{exm}\label{ex of decop of IVQ}
Given an ice valued quiver $Q:$
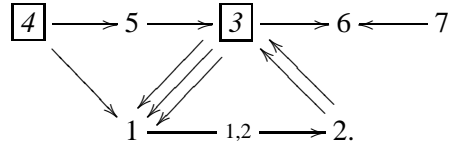
$$\xymatrix{\framebox{4}\ar[dr]\ar[r]&5\ar[r]&{\framebox{3}\ar[dl] \ar@<1ex>[dl] \ar@<-1ex>[dl]\ar[r]}&6&7\ar[l]
\\&1\ar[rr]|*+{\scriptstyle 1,2}&& 2.\ar[ul] \ar@<-1ex>[ul]&}$$
There are three indecomposable components of $Q$ by decomposing it at frozen vertices $3$ and $4$ as follows:
$$
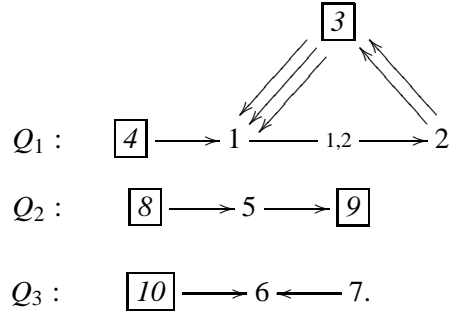
\xymatrix{&&{\framebox{3}\ar[dl] \ar@<1ex>[dl] \ar@<-1ex>[dl]}&\\
Q_1:~~~~~~\framebox{4}\ar[r]&1\ar[rr]|*+{\scriptstyle 1,2} && 2\ar[ul] \ar@<-1ex>[ul]}$$

$$\xymatrix{Q_2:~~~~~~~~\framebox{8}\ar[r]&5\ar[r]&{\framebox{9}}&}$$

$$\xymatrix{Q_3:~~~~~~~~\framebox{10}\ar[r]&6&7.\ar[l]&}$$
\end{exm}

The inverse process of the decomposition is the following gluing of ice valued quivers.
Given two ice valued quivers  $Q'=(Q'_0,Q'_1)$ and $Q''=(Q''_0,Q''_1)$, denote by $(Q'^{e}, v', d')$ and $(Q''^{e}, v'', d'')$ the principal parts of $Q'$ and $Q''$ respectively. Let $\Delta'$ and $\Delta''$ be two subsets of $Q'^f_0$ and $Q''^f_0$ respectively.
Assume that there is a bijection $f$ from $\Delta'$ to $\Delta''$ and we identify $\Delta'$ and $\Delta''$ as $\Delta$ under this bijection. We define a quiver $Q=(Q_0,Q_1)$, where $Q_0=(Q'_0\setminus \Delta')\sqcup (Q''_0\setminus \Delta'')\sqcup \Delta$ and the arrows in $Q_1$ come from $Q'_1$ and $Q''_1$. More precisely, there is an arrow $\alpha:i\longrightarrow j$ for $i,j \in Q_0$ if and only if there is an arrow $\alpha: i \longrightarrow j$ in $Q'_1$ or $Q''_1$.

\begin{prop}
The quiver $Q$ constructed above is an ice valued quiver where the principal part is a full subquiver of $Q$ corresponding to  ${Q'_0}^e\sqcup {Q''_0}^e $ and the functions $v$ and $d$ are given by:

$$v(\alpha) = \left\{\begin{array}{ll}
						v'(\alpha) &  \textrm{ if } \alpha \in {Q'_1}^e~\\
						v''(\alpha) &  \textrm{ if } \alpha \in {Q''_1}^e~;
					\end{array}\right.$$	

$$d_{i} = \left\{\begin{array}{ll}
						d'_{i} &  \textrm{ if } i \in {Q'_0}^e \\
						d''_{i} &  \textrm{ if } i \in {Q''_0}^e~.
					\end{array}\right.$$	
We call $Q$ the gluing of $Q'$ and $Q''$ along $f$ (or $\Delta$).
\end{prop}

\begin{proof}
This can be checked straightforward.
\end{proof}
Similarly, the gluing of extended skew-symmatrizable matrices can also be defined.

\subsection{Tensor decomposition and freezing}\label{Sec RCMT}

Recall that a seed is essentially determined by its exchange matrix, which is extended skew-symmetrizable. Then we decompose a seed as follows by decomposing the corresponding exchange matrix.
\begin{defn}\label{def of indec of seed}
Let $\S$ be a seed with the exchange matrix $B$. Denote by $\{B_1, B_2, \cdots, B_t\}$ the set of indecomposable components of $B$.
\begin{enumerate}
\item The seed $\S$ is said to be indecomposable if $B$ is indecomposable.
\item The decomposition of $\S$ is a set $Ind(\S)=\{\S_1, \S_2, \cdots, \S_t\}$, where each $\S_i$, $1 \leqslant i \leqslant t$, is a subseed of $\S$ corresponding to $B_i$. We call each $\S_i$, $1 \leqslant i \leqslant t$, an indecomposable component of $\S$. Note that $t$ could be a countable number.
\end{enumerate}
\end{defn}
As an inverse process of decomposition, we glue seeds at frozen variables by gluing the corresponding matrix as defined in the above subsection.
\begin{rem}
\begin{enumerate}
\item A cutting of a seed along separating families of frozen variables is defined in \cite{ADS13} (Definition 7.2).It is not hard to see that the decomposition of a seed at frozen variables is obtained by cuttings repeatedly.
\item It is also easy to see that an amalgamated sum along glueable seeds, which is defined in \cite{ADS13} (Definition 4.11), is in fact
 a special case of a glue of seeds defined above. Comparing with the amalgamated sum, in our sense, any two seeds are glueable along each bijection between some frozen variables in these two seeds since we don't care about connections between frozen variables.
\end{enumerate}
\end{rem}

From now on to the end of this subsection, we fix the following settings. Let $\Sigma=(\ex,\fx,B)$ be a seed with $\x=\ex \sqcup \fx$ and $B=(b_{xy})_{\x\times \ex}$. We denote by $Ind(\S)=\{\S_1,\S_2,\cdots ,\S_t\}$ the set of indecomposable components of $\S$,
where for each $1\leqslant i \leqslant t$, the seed $\S_i$ is $(\ex_i,\fx_i,B_i)$ with $\x_i=\ex_i\sqcup \fx_i$ and $B_i=(b^i_{xy})_{{\x_i}\times{\ex_i}}$, and $\X_{\S_i}$ is the set of cluster variables of $\S_i$.
Let $\A(\S_1)\otimes_\Z \A(\S_2)\otimes_\Z\cdots \otimes_\Z\A(\S_t)$ be the tensor algebra of the rooted cluster algebras $\A(\S_i)$, $1\leqslant i \leqslant t$. Then there is a quotient algebra $\A(\S_1)\otimes_\Z \A(\S_2)\otimes_\Z\cdots \otimes_\Z\A(\S_t)/\I$ by identifying the elements $1\otimes 1\cdots 1\otimes x^{i\textrm{-}th} \otimes 1 \cdots \otimes 1$
and $1\otimes 1\cdots 1\otimes x^{j\textrm{-}th} \otimes 1 \cdots \otimes 1$ for any $x \in \fx$ with $x \in \x_i\cap \x_j$ for some $1\leqslant i , j \leqslant t$. Now we state one of the main theorems in this subsection as follows:
\begin{thm}\label{inj of indec comp}

	\begin{enumerate}
     \item For each $\S_{i}\in Ind(\S)$, the canonical injection $j_i:\x_i \rightarrow \x$ induces an injective rooted cluster morphism from $\A(\S_i)$ to $\A(\S)$. Thus $\A(\S_i)$ is a rooted cluster subalgebra of $\A(\S)$.
     \item The set $\X_{\S}$ of cluster variables of $\S$ coincides with the disjoint union $\bigsqcup_{i=1}^{t}(\X_{\S_i}\setminus \fx_i)\bigsqcup \fx$.
     \item The ring homomorphism $$\tilde{j}:~~~\L_{\S_1,\Q}\otimes_\Q \L_{\S_2,\Q}\otimes_\Q\cdots \otimes_\Q\L_{\S_t,\Q} \rightarrow \L_{\S,\Q},$$ which is given by $x_1\otimes x_2\otimes \cdots \otimes x_t  \longmapsto j_{1}(x_1)j_{2}(x_2) \cdots j_{t}(x_t)$ for any $x_i \in \x_i (1\leqslant i \leqslant t)$, induces a ring isomorphism:
     $$j:~~~ \A(\S_1)\otimes_\Z \A(\S_2)\otimes_\Z\cdots \otimes_\Z\A(\S_{t})/\I \rightarrow \A(\S).$$
     We call $\A(\S_1)\otimes_\Z \A(\S_2)\otimes_\Z\cdots \otimes_\Z\A(\S_{t})/\I$ the tensor decomposition of $\A(\S)$.
        \end{enumerate}

\end{thm}

\begin{proof}
\begin{enumerate}
\item It is clear that the map $j_i:\x_i \rightarrow \x$ lifts as a ring homomorphism from $\L_{\S_i,\Q}$ to $\L_{\S,\Q}$ which satisfies $\MM 1$ and $\MM 2$. Note that each $\S_i$-admissible sequence $(x_1,x_2,\cdots ,x_l)$ is clearly $(j_i,\S_i,\S)$-biadmissible, thus to prove that $j_i$ induces a homomorphism from $\A(\S_i)$ to $\A(\S)$ and satisfies $\MM 3$, it is sufficient to show the equality $\mu_{x_l}\circ \cdots \circ\mu_{x_1,\S_i}(y)=\mu_{x_l}\circ \cdots \circ\mu_{x_1,\S}(y)$ for any $\S_i$-admissible sequence $(x_1,x_2,\cdots ,x_l)$ and any cluster variable $y\in \x_i$. Again, we prove the equality by induction on $l$. The case of $l=0$ is clear.
Assume that $l=1$. If $y\neq x_1$, then $\mu_{x_1,\S_i}(y)=y=\mu_{x_1,\S}(y)$. If $y=x_1$, then we have
$$\mu_{x_1,\S_i}(x_1)
=\frac{1}{x_1}\left(\prod_{\substack{x \in \x_i~; \\ b^i_{x{x_1}}>0}} x^{b^i_{x{x_1}}} + \prod_{\substack{x \in \x_i~; \\ b^i_{x{x_1}}<0}} x^{-b^i_{x{x_1}}}\right)
=\frac{1}{x_1}\left(\prod_{\substack{x \in \x~; \\ b_{x{x_1}}>0}} x^{b_{x{x_1}}} + \prod_{\substack{x \in \x~; \\ b_{x{x_1}}<0}} x^{-b_{x{x_1}}}\right)
=\mu_{x_1,\S}(x_1)$$ where the second equality follows from the definition of the indecomposable component. Thus $\mu_{x_1,\S_i}(y)=\mu_{x_1,\S}(y)$ for any cluster variable $y\in \x_i$.
When $l=2$, it is clear that $\mu_{x_1}(\S_i)$ is an indecomposable component of $\mu_{x_1}(\S)$.
Then similar to the case of $l=1$, it can be shown that for any cluster variable $\mu_{x_1}(y) \in \mu_{x_1}(\S_i)$, $\mu_{x_2,\mu_{x_1}(\S_i)}(\mu_{x_1}(y))=\mu_{x_2,\mu_{x_1}(\S)}(\mu_{x_1}(y))$, thus we have $\mu_{x_2}\circ\mu_{x_1,\S_i}(y)=\mu_{x_2,\mu_{x_1}(\S_i)}(\mu_{x_1}(y))=\mu_{x_2,\mu_{x_1}(\S)}(\mu_{x_1}(y))=\mu_{x_2}\circ\mu_{x_1,\S}(y)$. Finally, we inductively prove the equality in this way.

\item It follows from the proof of the above assertion that the set $\X_{\S}$ contains the set $\bigsqcup_{i=1}^{t}(\X_{\S_i}\setminus \fx_i)\bigsqcup \fx$. Note that any cluster variable of $\A(\S)$ is of the form $\mu_{x_l}\circ \cdots \circ\mu_{x_1,\S}(y)$, where $(x_1,x_2,\cdots ,x_l)$ is a $\S$-admissible sequence and $y \in \x$ is a cluster variable of $\S$. Assume that $y \in \x_i$ and
    $\{x_{i_t}\}_{t=1}^s$ be the maximal subset of $\{x_k\}_{k=1}^l$ with $x_{i_t}\in \x_i$ and $i_{t}
    < i_{t'}$ for each $1\leqslant t < t' \leqslant s$. Then we have $\mu_{x_l}\circ \cdots \circ\mu_{x_1,\S}(y)=\mu_{x_{i_s}}\circ \cdots \circ\mu_{x_{i_1},\S_i}(y)$ since $\S_i$ is an indecomposable component of $\S$. Thus the inverse inclusion is valid and $\X_{\S}=\bigsqcup_{i=1}^{t}(\X_{\S_i}\setminus \fx_i)\bigsqcup \fx$.

\item Firstly, it is easy to see that $\tilde{j}$ is a ring homomorphism and induces a surjective ring homomorphism $\hat{j}$ from $\A(\S_1)\otimes_\Z \A(\S_2)\otimes_\Z\cdots \otimes_\Z\A(\S_{t})$ to $\A(\S)$ due to the assertion 2. Secondly, the kernel $Ker(\hat{j})$ of $\hat{j}$ is $\I$ and thus we get the conclusion. In fact, $Ker(\hat{j})=\I$ is because $Ker(\tilde{j}) = \I$.
\end{enumerate}

\end{proof}

\begin{exm}
Consider the ice valued quiver $Q$ and its indecomposable components $Q_i$ ($1\leqslant i \leqslant3$) in Example \ref{ex of decop of IVQ}. Denote by $\A$ and $\A_i$ ($1\leqslant i \leqslant3$) the rooted cluster algebras corresponding to $Q$ and $Q_i$ ($1\leqslant i \leqslant3$) respectively. The ambient rings of $\A$ and $\A_i$ ($1\leqslant i \leqslant3$) are respectively as follows:
$$~~~~~~~\L=\Q[x_1^{\pm1}, x_2^{\pm1}, x_5^{\pm1}, x_6^{\pm1}, x_7^{\pm1}, x_3, x_4],$$
$$\L_1=\Q[x_1^{\pm1}, x_2^{\pm1}, x_3, x_4],~~~~~~~~~~~~$$
$$\L_2=\Q[x_5^{\pm1}, x_8, x_9],~~~~~~~~~~~~~~~~~~~$$
$$\L_3=\Q[x_6^{\pm1}, x_7^{\pm1}, x_{10}].~~~~~~~~~~~~~~~~$$
Let $j: \L_1\otimes_\Q \L_2\otimes_\Q \L_3 \rightarrow \L$ be a ring homomorphism with
$$j(x_i)= \left\{\begin{array}{rcl}
				x_i, & & i=1, 2, 5, 6, 7; \\
				x_3, & & i=3, 9, 10; \\
				x_4, & & i=4, 8.
               \end{array}\right.$$
Then from above theorem, it induces a ring isomorphism $j: \A_1\otimes_\Z \A_2\otimes_\Z\A_3/\I \rightarrow \A$, where $\I=<x_3 \otimes 1 \otimes 1 - 1 \otimes x_9 \otimes 1, x_3 \otimes 1 \otimes 1 - 1 \otimes 1 \otimes x_{10}, x_4 \otimes 1 \otimes 1 - 1 \otimes x_8 \otimes 1>$.
\end{exm}

\begin{defn}\label{def of frozen}
Let $\Sigma=(\ex,\fx,B)$ be a seed and $\ex_0$ be a subset of $\ex$. The seed $\Sigma_f=(\ex \setminus \ex_0,\fx\sqcup \ex_0,B_f)$ is called the freezing of $\S$ at $\ex_0$, where $B_f$ is obtained from $B$ by deleting the columns corresponding to the elements in $\ex_0$.
\end{defn}

\begin{exm}
Consider the seed
\begin{center}$\S = \left((x_1,x_2,x_3),\varnothing,
				\left[\begin{array}{rrr}
					0 & -2 & 6 \\
					1 & 0 & -3 \\
					-2 & 2 & 0
				\end{array}\right]
			\right) ~~~~~\text{and its ice valued quiver}~~~~~$
~\\
~\\
$
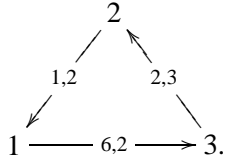
\begin{xy} 0;<0.5pt,0pt>:<0pt,-0.5pt>::
(-75,100) *+{1} ="0",
(0,0) *+{2} ="1",
(75,100) *+{3.} ="2",
"1", {\ar|*+{\scriptstyle 1,2}"0"},
"0", {\ar|*+{\scriptstyle 6,2}"2"},
"2", {\ar|*+{\scriptstyle 2,3}"1"},
\end{xy}$
\end{center}
Let
$$\S_f = \left((x_1,x_2),(x_3),
				\left[\begin{array}{rr}
					0 & -2 \\
					1 & 0 \\
					-2 & 2
				\end{array}\right]
			\right)$$
be the freezing of $\S$ at the exchangeable cluster variable $x_3$. Then the corresponding ice valued quiver is
$$
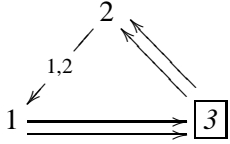
\xymatrix{&2{\ar[dl]|*+{\scriptstyle 1,2}}&
\\1\ar[rr] \ar@<-1ex>[rr] && \framebox{3} \ar[ul] \ar@<-1ex>[ul]}$$
where we have framed the frozen vertex.
\end{exm}

\begin{defn-prop}\label{inj of frozen}
The natural bijection $j$ from the cluster variables of $\S_f$ to the cluster variables of $\S$ induces an injective rooted cluster morphism from $\A(\S_f)$ to $\A(\S)$. Thus $\A(\S_f)$ is a rooted cluster subalgebra of $\A(\S)$. We call $\A(\S_f)$ the freezing of $\A(\S)$ at $\ex_0$.
\end{defn-prop}

\begin{proof}
Firstly, it is clear that $j$ lifts as an injective ring homomorphism from $\L_{\S_f,\Q}$ to $\L_{\S,\Q}$ which satisfies $\MM 1$ and $\MM 2$. Secondly, each $\S_f$-admissible sequence $(x_1,x_2, \cdots,x_l)$ is clearly $(j,\S_f,\S)$-biadmissible.
Then similar to the proof of Theorem \ref{inj of indec comp}(1), $\MM 3$ is also valid. So we are done.
\end{proof}

\subsection{Injections and specializations}\label{Sec RCMIS}

In this subsection, we characterize injective rooted cluster morphisms and define the general specialization.

\medskip

Given two seeds $\S_0$ and $\S$. Let $g: \A(\S_0)\longrightarrow\A(\S)$ be an injective rooted cluster morphism. Then from the injectivity, $g$ induces an injection from the set $\x_{\S_0}$ of initial cluster variables in
$\A(\S_0)$ to the set $\x_{\S}$ of initial cluster variables in $\A(\S)$.
In fact, for any $x\in \x_{\S_0}$, if $g(x)=n \in \Z$, then we have $g(n)=g(x)=n$ and a contradiction to the injectivity of $g$.
Thus there is no harm to assume that $\S_0=(\ex_0,\fx_0\sqcup \ex_2,B_0)$ and $\S=(\ex_0\sqcup \ex_1 \sqcup \ex_2,\fx_0\sqcup \fx_1,B)$ where $g$ is an identity on $\x_{\S_0}=\ex_0\sqcup \fx_0\sqcup \ex_2$. We define $\S_f=(\ex_0\sqcup \ex_1,\fx_0\sqcup \fx_1 \sqcup \ex_2,B_f)$ as the freezing of $\S$ at $\ex_2$ and $\S_1=(\ex_1,\fx_1\sqcup \ex_2,B_1)$ as a subseed of $\S_f$. Denote by $Ind(\S_0)$ ($Ind(\S_1)$ and $Ind(\S)$ respectively) be the set of indecomposable components of $\S_0$ ($\S_1$ and $\S$ respectively).

\begin{thm}\label{Prop of inj}
Let $\A(\S_0)$ be a rooted cluster subalgebra of $\A(\S)$ under an injective rooted cluster morphism $g: \A(\S_0)\longrightarrow\A(\S)$ with notations as above.
\begin{enumerate}
     \item If the seed $\S_0$ is indecomposable, then $\S_0$ (or $\S^{op}_0$) is an indecomposable component of $\S_f$, and
     the seed $\S_f$ is obtained by gluing $\S_0$ (or $\S^{op}_0$) and $\S_1$ along $\ex_2$.
     \item There is a bijection $h$ from $Ind(\S_0)\sqcup Ind(\S_1)$ to $Ind(\S_f)$ such that $h(\S')=\S'$ or $h(\S')={\S'}^{op}$ for each $\S'\in Ind(\S_0)$ and $h(\S')=\S'$ for each $\S'\in Ind(\S_1)$ .
     \item There is a rooted cluster isomorphism $\A(\S_0)\cong \A(g(\S_0))$.
     \item The injection $g$ is a section if and only if  ~$\ex_2$~ is an empty set, that is, $\S_0$ is a subseed of $\S$.
\end{enumerate}
\end{thm}

\begin{proof}
\begin{enumerate}
     \item For each exchangeable cluster variable $x\in \ex_0$, we have $${\mu_x}(x)=\frac{1}{x}\left(\prod_{\substack{y \in \x_{\S_0}~; \\ b^0_{yx}>0}} y^{b^0_{yx}} + \prod_{\substack{y \in \x_{\S_0}~; \\ b^0_{yx}<0}} y^{-b^0_{yx}}\right)$$
in $\A(\S_0)$, where $b^0_{yx}$ entries in $B_0$.
Then by viewing $g$ as an identity on the set $\x_{\S_0}$, due to $\MM 3$, we have the equalities:
$$\frac{1}{x}\left(\prod_{\substack{y \in \x_{\S_0}~; \\ b^0_{yx}>0}} y^{b^0_{yx}} + \prod_{\substack{y \in \x_{\S_0}~; \\ b^0_{yx}<0}} y^{-b^0_{yx}}\right)=g({\mu_x}(x))={\mu_{g(x)}}(g(x))=\frac{1}{x}\left(\prod_{\substack{y \in \x_{\S_f}~; \\ b_{yx}>0}} y^{b_{yx}} + \prod_{\substack{y \in \x_{\S_f}~; \\ b_{yx}<0}} y^{-b_{yx}}\right),$$ where $b_{yx}$ entries in $B$ and $\x_{\S_f}(=\x_{\S} ~\textrm{as a set})$ is the set of cluster variables of $\S_f$.
Thus a cluster variable $y$ of $\S_f$ is connected with $x$ in $\S_f$ if and only if it belongs to $\x_{\S_0}$ and is connected with $x$ in $\S_0$, and moreover, $b^0_{yx}=b_{yx}$ or $b^0_{yx}=-b_{yx}$. Finally, one of the two seeds $\S_0$ and $\S^{op}_0$ is a subseed of $\S_f$ by the connectivity of $\S_0$. Again by the connectivity, one of these two seeds is an indecomposable component of $\S_f$.

\item For each indecomposable component $\S'$ in $Ind(\S_0)$, the composition of the morphism $g$ and the natural injection from $\A(\S')$ to $\A(\S_0)$ induces an injective rooted cluster morphism from $\A(\S')$ to $\A(\S)$. Thus $\S'$ or $\S'^{op}$ can be viewed as an indecomposable component of $\S_f$ from the first statement. It is clear that $Ind(\S_1)$ is a subset of $Ind(\S_f)$, thus we have an injection $h: Ind(\S_0)\sqcup Ind(\S_1)\longrightarrow Ind(\S)$ satisfies the properties we want. For the bijectivity, it is only need to notice that the set of exchangeable cluster variables arising from the seeds in $Ind(\S_0)\sqcup Ind(\S_1)$ contains the set of exchangeable cluster variables of $\S_f$.

\item This assertion is clear.

\item This easily follows from the fact that it is forbidden in a rooted cluster morphism from mapping an initial exchangeable variable to a frozen one.
\end{enumerate}
\end{proof}

\begin{rem}
\begin{enumerate}\label{rem of inj RCM}
\item The above theorem allows us to find out all the rooted cluster subalgebras of a given rooted cluster algebra up to isomorphism. The third part shows that the rooted cluster subalgebra in the sense of Definition \ref{def of injRCM} coincides with the subcluster algebra defined in section \uppercase\expandafter{\romannumeral4}.1\cite{BIRS09} up to a rooted cluster isomorphism.
\item The fourth part of the above theorem answers the Problem 7.7 in \cite{ADS13}; more precisely, an injective rooted cluster morphism $g$ is a section if and only if $g(\S_0)$ is a subseed of $\S'$.
\end{enumerate}
\end{rem}

The following two theorems describe the relations between ideal rooted cluster morphisms and injective rooted cluster morphisms.
\begin{thm}\label{image of subalgebra}
Given two seeds $\Sigma=(\ex,\fx,B)$ and $\Sigma'=(\ex',\fx',B')$ with  $\x=\ex \sqcup \fx$ and $\x'=\ex' \sqcup \fx'$. Let $f: \A(\S)\rightarrow \A(\S')$ be a rooted cluster morphism.
\begin{enumerate}
\item The rooted cluster algebra $\A(f(\S))$ is a rooted cluster subalgebra of $\A(\S')$ under the natural injection on the initial cluster variables.
\item If $f$ is ideal then it is a composition of a surjective rooted cluster morphism and an injective rooted cluster morphism.
\end{enumerate}
\end{thm}

\begin{proof}
\begin{enumerate}
\item We consider the image seed $f(\S)=(\ex'\cap f(\ex),{\x'\cap f(\x)}\setminus {\ex'\cap f(\ex)},B'[\x'\cap f(\x)])$ and the freezing $\S'_f=(\ex' \setminus \ex'_0, \fx' \sqcup \ex'_0, B_f')$ of $\S'$ at $\ex'_0$, where $\ex'_0=\{f(x)\in \ex' | x\in \fx\}$. Then by a same argument used in the proof of Theorem \ref{ideal morp}(1), there exists no $y \in \ex$ such that $f(y)\in \ex'_0$. Thus $f(\ex) \cap \ex'_0=\emptyset$ and $\ex'\cap f(\ex) \subseteq \ex' \setminus \ex'_0$, ${\x'\cap f(\x)}\setminus {\ex'\cap f(\ex)} \subseteq \fx' \sqcup \ex'_0$. Therefore the seed $f(\S)$ is a subseed of $\S'_f$. Given an exchangeable variable $x_1$ of $f(\S)$, we can assume that $x_1=f(x)$ for some $x \in \ex$. Then we have the following equalities:

$$f(\mu_{{x},\S}({x}))=f\left(\frac{1}{x}\left(\prod_{\substack{y \in \x~; \\ b_{yx}>0}} y^{b_{yx}} + \prod_{\substack{y \in \x~; \\ b_{yx}<0}} y^{-b_{yx}}\right)\right)=\frac{1}{x_1}\left(\prod_{\substack{y \in \x~; \\ b_{yx}>0}} f(y)^{b_{yx}} + \prod_{\substack{y \in \x~; \\ b_{yx}<0}} f(y)^{-b_{yx}}\right),$$

$$\mu_{x_1,\S'}(x_1)
=\frac{1}{x_1}\left(\prod_{\substack{y_1 \in \x'~; \\ b'_{{y_1}{x_1}}>0}} y_1^{ b'_{{y_1}{x_1}}} + \prod_{\substack{y_1 \in \x'~; \\ b'_{{y_1}{x_1}}<0}} y_1^{- b'_{{y_1}{x_1}}}\right)~~~and~~~$$

$$\mu_{x_1,f(\S)}(x_1)
=\frac{1}{x_1}\left(\prod_{\substack{y_1 \in \x'\cap f(\x)~; \\ b'_{{y_1}{x_1}}>0}} y_1^{ b'_{{y_1}{x_1}}} + \prod_{\substack{y_1 \in \x'\cap f(\x)~; \\ b'_{{y_1}{x_1}}<0}} y_1^{- b'_{{y_1}{x_1}}}\right).$$
Because $f$ is a rooted cluster morphism, we have $f(\mu_{{x},\S}({x}))=\mu_{x_1,\S'}(x_1)$. Then by a similar trick used in the proof of Theorem \ref{ideal morp}(1) we get $\mu_{x_1,f(\S)}(x_1)=\mu_{x_1,\S'}(x_1)$. Thus for each $y_1\in \x'$ with $b_{y_1x_1} \neq 0$, we have $y_1\in \x' \cap f(\x)$. Therefore, $f(\S)$ is a glue of some indecomposable components of $\S'_f$. Finally, $\A(f(\S))$ is a rooted cluster subalgebra of $\A(\S'_f)$ and thus a rooted cluster subalgebra of $\A(\S')$.
\item Because $f$ is ideal, it induces a surjective rooted cluster morphism $f': \A(\S) \rightarrow \A(f(\S))$. It follows from the above assertion
    that $j: \A(f(\S)) \rightarrow \A(\S')$ is an injective rooted cluster morphism. Thus $f$ is a composition of rooted cluster morphisms $f'$ and $j$.
\end{enumerate}
\end{proof}

\begin{thm}\label{Inverse of rem of ideal}
Let $f:\A(\S) \rightarrow \A(\S')$ be an inducible rooted cluster morphism. Assume that it is a composition of rooted cluster morphisms $f':\A(\S) \rightarrow \A(\S'')$ and $f'':\A(\S'') \rightarrow \A(\S')$, where $f'$ is surjective and $f''$ is injective. Then $\A(\S'')\cong \A(f(\S))$ and $f$ is ideal.
\end{thm}

\begin{proof}
Note that $f'$ is a surjective inducible rooted cluster morphism because $f$ is inducible.
Then from Theorem \ref{ideal morp}, $f'$ is ideal and thus $f'(\A(\S))=\A(f'(\S))$. On the one hand, $f'(\S)=\S''$ from Lemma 3.1 in \cite{ADS13}. On the other hand, we have $f''(\A(\S''))=\A(f''(\S''))$ because $f''$ is injective and thus ideal due to Corollary 4.5 in \cite{ADS13}. Therefore there are equalities  $f(\A(\S))=f''f'(\A(\S))=f''(\A(f'(\S)))=f''(\A(\S''))=\A(f''(\S''))$. Thus to prove that $f$ is ideal, that is, $f(\A(\S)=\A(f(\S))$, it is only need to show that $\A(f''(\S''))\subseteq \A(f(\S))$. Let $x \in \x_{f''(\S'')}$ be a cluster variable of $f''(\S'')$ and assume that $y\in \x_{\S''}$ is a cluster variable such that $f''(y)=x$. By Lemma 3.1 in \cite{ADS13}, there is a cluster variable $y' \in \x_{\S}$ such that $f'(y')=y$. Then $x=f(y')\in \x_{f(\S)}$ and thus $\x_{f''(\S'')} \subseteq \x_{f(\S)}$. Similarly, we have $\ex_{f''(\S'')} \subseteq \ex_{f(\S)}$. Therefore $\A(f''(\S''))\subseteq \A(f(\S))$ and $f$ is ideal. The rooted cluster isomorphism $\A(\S'')\cong \A(f(\S))$ is clear.
\end{proof}

For a rooted cluster algebra, we introduce the following notion of complete pairs of subalgebras with given coefficients.
\begin{defn}\label{def of complete pairs}
Let $\S=(\ex, \fx, B)$ be a seed and $\S_f$ be the freezing of $\S$ at a subset $\ex'$ of $\ex$. Denote by $\fx_0$ the set of isolated frozen variables of $\S_f$. Let $\S_1=(\ex_1, \fx_1, B_1)$ and $\S_2=(\ex_2, \fx_2, B_2)$ be two subseeds of $\S_f$. We call the pair $(\A(\S_1), \A(\S_2))$ a complete pair of subalgebras of $\A(\S)$ with coefficient set $\fx \sqcup \ex'$ if the following conditions are satisfied:
\begin{enumerate}
\item both $\S_1$ and $\S_2$ are gluings of some indecomposable components of $\S_f$;
\item $\ex_1\cap \ex_2 = \emptyset$ and $\ex = \ex_1 \sqcup \ex_2 \sqcup \ex'$;
\item $\fx_0\subseteq \fx_1\cap\fx_2$.
\end{enumerate}
\end{defn}

In the rest of this subsection, we consider the specialization of a seed at some cluster variables. Let $\S=(\ex,\fx,B)$ be a seed with $\x=\ex\sqcup \fx$. Let $\x'=\ex'\sqcup \fx'$ be a subset of $\x$ with $\ex'\subseteq \ex$ and $\fx'\subseteq \fx$. Denote by $\S'=(\ex',\fx',B')$ the subseed of $\S$.

\begin{defn}\label{def of spec}
The map $f: \x \rightarrow \x'\sqcup \Z$ is called a specialization of $\S$ at $\x\setminus\x'$, which is given by:
$$f(x) = \left\{\begin{array}{ll}
						x &  \textrm{ if } x \in \x'~;\\
						n &  \textrm{ if } x \in {\x\setminus\x'}~.
					\end{array}\right.$$	
If the map $f$ induces a rooted cluster morphism $\tilde{f}$ from $\A(\S)$ to $\A(\S')$, then we call $\tilde{f}$ the specialization of rooted cluster algebra $\A(\S)$ at $\x\setminus\x'$. If $\x\setminus\x'=\{x\}$ for some $x$, then we call $f$ a simple specialization.
\end{defn}

The simple specialization is also defined in \cite{ADS13} and it is needed that $f(x)\neq 0$ for $x \in \x \setminus \x'$. It is well-known that if $x$ is frozen, then specializing $x$ to 1 induces a rooted cluster morphism from $\A(\S)$ to $\A(\S')$, see for instance \cite{FZ07}. Generally we have the following analogue:

\begin{prop}\label{prop of spec for frozen}
Let $f$ be a specialization defined as above. Assume that $\ex'=\ex$ and $f(x)=1$ for all cluster variables $x\in {\x\setminus\x'}$, then the map $f$ induces a rooted cluster morphism $\tilde{f}$ from $\A(\S)$ to $\A(\S')$.
\end{prop}

\begin{proof}
Similar to the case of simple specialization defined in \cite{ADS13}.
\end{proof}

\section{Relations with cluster structures in $2$-Calabi-Yau triangulated categories}\label{Sec Tri}

In this section, we consider the cluster substructure of a functorially finite extension-closed subcategory in a $2$-Calabi-Yau triangulated category which has a cluster structure given by cluster tilting subcategories. We establish a connection between functorially finite extension-closed subcategories, cluster substructures and rooted cluster subalgebras.

\subsection{Preliminaries of $2$-Calabi-Yau triangulated categories}\label{Sec TriP}

Let $k$ be an algebraically closed field. Let $\C$ be a triangulated category, and we always assume that it is $k$-linear, Hom-finite and Krull-Schmidt. Denote by $[1]$ the shift functor in $\C$. $\C$ is called $2$-Calabi-Yau if there is a bifunctorial isomorphism $\Ext^1_{\C}(X, Y)\cong D\Ext^1_{\C}(Y, X)$ for all objects $X$ and $Y$ in $\C$, where $\textrm{D}=\Hom_{k}(-,k)$ is the $k$-duality. For a subcategory $\B$ of $\C$, we always mean that $\B$ is a full additive category which closed under direct sums, direct summands and isomorphisms. Denote by $ind\B$ the set of isomorphism classes of indecomposable objects in $\B$. For an object $X$ in $\C$, we denote by $add X$ the subcategory additive generated by $X$. For a set $\mathcal{X}$ of isomorphism classes of indecomposable objects in $\C$, we denote by $add \mathcal{X}$ the subcategory additive generated by $\mathcal{X}$. In this way, we have for a subcategory $\B$, $\B =add(ind\B)$. For subcategories $\B$ and $\D$ with $ind \B\bigcap ind\D=\emptyset$, we denote by $\B\oplus \D$ the subcategory $add(ind\B\cup ind\D)$. A subcategory $\B$ of $\C$ is called contravariantly finite if for each object $C$ in $\C$, there exist an object $B \in \B$ and a morphism $f \in \Hom_{\C}(B,C)$ such that the morphism $f\cdot :\Hom_{\C}(-,B)\rightarrow \Hom_{\C}(-,C)$ is surjective on $\B$, where $f\cdot$ maps any morphism $g \in \Hom_{\C}(B',B)$ to $fg$ for each $B' \in \B$. Here, $f$ is called a right $\B$-approximation of the object $C$. Moreover, if $f: B\rightarrow C$ has no direct summand of the form $D \rightarrow 0$ as complex, we say $f$ right minimal. A convariantly finite subcategory, a left approximation and a left minimal map are defined in a dual way. If a subcategory is both contravariantly finite and convariantly finite, then we call it functorially finite. A functorially finite subcategory $\T$ of $\C$ is called a cluster tilting subcategory if the following conditions are satisfied:

\begin{enumerate}
\item $\T$ is rigid, that is, $\Ext^1_{\C}(T, T')=0$ for all $T, T' \in \T$;
\item For an object $X\in \C$, $X\in \T$ if and only if $\Ext^1_{\C}(X,T)=0$ for any $T\in \T$.
\end{enumerate}
An object $T$ in $\C$ is called a cluster tilting object if $\T=add T$ is a cluster tilting subcategory. The cluster tilting subcategories in a $2$-Calabi-Yau triangulated category have many remarkable properties. For example the property given in the following
\begin{lem} (Theorem 5.3 \cite{IY08})
Let $\T$ be a cluster tilting subcategory in a $2$-Calabi-Yau triangulated category $\C$ and $T$ be an indecomposable object in $\T$. Then there is a unique (up to isomorphism) indecomposable object $T'\ncong T$ in $\C$ such that ${\mu_T}(\T):= add(ind\T \setminus \{T\} \cup \{T'\})$ is a cluster tilting subcategory in $\C$, where $T'$ and $T$ form two triangles:
$${T}\s{f}\longrightarrow E\s{g}\longrightarrow {T'}\s{}\longrightarrow {T[1]}~~~~~~~~~~~~ and$$
$${T'}\s{s}\longrightarrow E'\s{t}\longrightarrow {T}\s{}\longrightarrow {T'[1]}~~~~~~~~~~~~~~~~~~$$
with minimal right $add(ind\T \setminus \{T\})$-approximation $g$ and $t$ and minimal left $add(ind\T \setminus \{T\})$-approximation $f$ and $s$.
\end{lem}

The subcategory $\mu_{T}(\T)$ is the mutation of $\T$ at $T$. The first triangle is the right exchange triangle of $T$ in $\T$ and the second one is the left exchange triangle of $T$ in $\T$. We say that a cluster tilting subcategory $\T'$ is reachable from $\T$, if it can be obtained by a finite sequence of mutations from $\T$. A rigid object in $\C$ is called reachable from $\T$ if it belongs to a cluster tilting subcategory which is reachable from $\T$. Denote by $\R(\T)$ the subcategory of $\C$ additive generated by rigid objects in $\C$ which are reachable from $\T$.

\medskip

For a cluster tilting subcategory $\T$, one can associate it a quiver $Q(\T)$ as follows: the vertices of $Q(\T)$ are the isomorphism classes of indecomposable objects in $\T$ and the number of arrows from $T_i$ to $T_j$ is given by the dimension of the space $irr(T_i,T_j)=rad(T_i,T_j)/rad^2{(T_i,T_j)}$ of irreducible morphisms , where $rad( , )$ is the radical in $\T$.
The existence of the exchange triangles for each indecomposable object in $\T$ guarantees that $Q(\T)$ is a locally finite quiver.

\medskip

The cluster structure of a triangulated category is defined in \cite{BIRS09} and also be studied in \cite{FK10}. For the general definition of a cluster structure, we refer to \cite{BIRS09}. In our settings, we use Theorem II 1.6 in \cite{BIRS09} to define the cluster structure.

\begin{defn}
We say that $\C$ has a cluster structure given by its cluster tilting subcategories if the quiver $Q(\T)$ has no loops nor $2$-cycles for each cluster tilting subcategory $\T$ of $\C$.
\end{defn}

We recall the definitions of cotorsion pairs, t-structures and relative cluster tilting subcategories in cotorsion pairs from \cite{IY08,Na11,ZZ12,BBD81,BR07}.
\begin{defn}
\begin{enumerate}
\item A pair $(\X,\Y)$ of functorially finite subcategories of $\C$ is called a cotorsion pair if $\Ext^1_{\C}(\X,\Y)=0$ and $$\C=\X \ast \Y[1]:=\{Z \in \C ~|~ \exists ~ \textrm{a triangle}~ X\s{}\longrightarrow Z\s{}\longrightarrow Y[1]\s{}\longrightarrow X[1]
~ \textrm{in}~ \C ~ \textrm{with}~ X \in \X, Y \in \Y \}.$$
We call the subcategory $\X$ ($\Y$ respectively) the torsion subcategory (torsionfree subcategory respectively) of $(\X,\Y)$ and $\I=\X\cap \Y$ the core of $(\X,\Y)$.
\item \cite{BBD81,BR07} A cotorsion pair $(\X,\Y)$ in $\C$ is called a t-structure if $\X$ is closed under [1](equivalently $\Y$ is closed under [-1]).
\item \cite{ZZ12} Let $\X$ be a subcategory of $\C$. A functorially finite subcategory $\mathcal{D}$ of $\X$ is called a $\X$-cluster tilting subcategory provided that for an object $D$ of $\X$, $D \in \D$ if and only if $\Ext^1_{\C}(X,D)=0$ (thus $\Ext^1_{\C}(D,X)=0$) for any object $X \in \X$.
\end{enumerate}
\end{defn}

\begin{rem}
\begin{enumerate}
\item The conditions in the definition of cotorsion pairs $(\X,\Y)$ above is stronger than the usual one \cite{IY08}), here we always assume that $\X$ and $\Y$ are functorially finite. For $2-$Calabi-Yau triangulated categories with cluster tilting objects, it was proved that for any cotorsion pairs  $(\X,\Y)$ in the usual sense \cite{IY08},  $\X$ and $\Y$ are functorially finite.
\item It can be easily derived from the $2$-Calabi-Yau property of $\C$ that if $(\X,\Y)$ is a cotorsion pair of $\C$, then $(\Y,\X)$ is also a cotorsion pair. Due to the symmetry, $\X$ and $\Y$ possess similar properties. Thus for the convenience, we only state definitions and properties about torsion subcategory $\X$ in the following.
\end{enumerate}
\end{rem}

\subsection{Cluster structures in subfactor triangulated categories}\label{Sec TriSubfactor}
Let $\I$ be a functorially finite rigid subcategory in $\C$.
Then the subfactor category $\C'=^{\bot}\I[1]/\I$ is a $2$-Calabi-Yau triangulated category with triangles and cluster tilting subcategories induced from triangles and cluster tilting subcategories of $\C$ respectively (Theorem 4.2 in \cite{IY08}). For an object $X\in {^{\bot}\I[1]}$, we denote also by $X$ the object in the quotient category $\C'$. For a morphism $f\in \Hom_\C(X_1,X_2)$ with $X_1, X_2\in {^{\bot}\I[1]}$, we denote by $\underline{f}$ the residue class of $f$ in the quotient category $\C'$.
Now assume that $\C$ has a cluster structure, we show in this subsection that the subfactor category $^{\bot}\I[1]/\I$ inherits the cluster structure from $\C$.

\begin{lem}\label{lem in 3.2}
Let $\T$ be a cluster tilting subcategory in $\C'$ and $T$ be an indecomposable object in $\T$. Then $\T\oplus\I$ is a cluster tilting subcategory of $\C$.
\begin{enumerate}
\item Let triangles
$${T}\longrightarrow E\longrightarrow {T'}\longrightarrow {T\langle1\rangle}~~~~~~~~~~~~ and$$
$${T'}\longrightarrow E'\longrightarrow {T}\longrightarrow {T'\langle1\rangle}~~~~~~~~~~~~~~~~~~$$
be the exchange triangles of $T$ in $\T$, where $\langle1\rangle$ is the shift functor in $\C'$.
Then in the triangulated category $\C$ the exchange triangles of $T$ in $\T\oplus\I$ are of the forms
$$\divideontimes ~~~~~~~~~~~~~~~~~{T}\longrightarrow E\oplus I\longrightarrow {T'}\longrightarrow {T[1]}~~~~~~~~~~~ and~~~~~~~~~~~~~~~~~~~~~~~~$$
$${T'}\longrightarrow E'\oplus I'\longrightarrow {T}\longrightarrow {T'[1]},~~~~~~~~~~~~~~~~~~$$
where $I$ and $I'$ are both in $\I$.
\item Let triangles
$${T}\longrightarrow E\oplus I\longrightarrow {T'}\longrightarrow {T[1]}~~~~~~~~~~~ and~~~~$$
$${T'}\longrightarrow E'\oplus I'\longrightarrow {T}\longrightarrow {T'[1]}~~~~~~~~~~~~~~~~~~$$
be the exchange triangles of $T$ in $\T\oplus\I$, where $I$ and $I'$ are both in $\I$, and $E$ and $E'$ have no direct summands in $\I$.
Then in the triangulated category $\C'$, they respectively induce
$${T}\longrightarrow E\longrightarrow {T'}\longrightarrow {T\langle1\rangle}~~~~~~~~~~~~ and$$
$${T'}\longrightarrow E'\longrightarrow {T}\longrightarrow {T'\langle1\rangle}~~~~~~~~~~~~~~~~~~$$
as the exchange triangles of $T$ in $\T.$
\end{enumerate}
\end{lem}
\begin{proof}
\begin{enumerate}
\item It follows from \cite{IY08}(Theorem 4.9) that $\T\oplus\I$ is a cluster tilting subcategory of $\C$. Because ${T}\s{\underline{a_1}}\longrightarrow E\longrightarrow {T'}\longrightarrow {T\langle1\rangle}$ is a triangle in $\C'$, from the triangulated structure of $\C'$ (see section 4 in \cite{IY08}), we have the following commutative diagram of morphisms between triangles in $\C$
$$\begin{array}{ccccccc}
\bigstar~~~~~~~~~~~~~~~~~T& \s{(a_1~ a_2)}\longrightarrow&E\oplus I_1&\longrightarrow&T'\oplus I_2&\longrightarrow&T[1]~~~~\\
~~~~~~~~~~~~~~~~~~~~~~\parallel&&~~~~~~~\downarrow \begin{pmatrix} b_1\\b_2 \end{pmatrix}&&\downarrow&&\parallel~~~~\\
~~~~~~~~~~~~~~~~~~~~~~T&\s{\alpha}\longrightarrow&I_0&\longrightarrow&{T\langle1\rangle}&\longrightarrow&T[1],~~~~\\
\end{array}$$
where $\alpha$ is a left $\I$-approximation in $\C$, and $I_1$ and $I_2$ are in $\I$.
We claim that $(a_1~ a_2)$ is a left $add(ind(\T\oplus\I) \setminus \{T\})$-approximation in $\C$. Let $T_0$ be an indecomposable object in $add(ind(\T\oplus\I) \setminus \{T\})$ and $f$ be a morphism in $\Hom_\C(T, T_0)$. If $T_0 \in \I$, then there exists $g\in \Hom_\C(I_0, T_0)$ such that $f=g\alpha$ because $\alpha$ is a left $\I$-approximation. Thus we have $f=ga_1b_1+ga_2b_2$, that is, $f$ factor through $(a_1~ a_2)$. If $T_0 \in add(ind\T \setminus \{T\})$, then there exists $\underline{h}\in \Hom_{\C'}(E, T_0)$ such that $\underline{f}=\underline{h}\underline{a_1}$ since $\underline{a_1}$ is a left $add(ind\T \setminus \{T\})$-approximation in $\C'$. Thus there exists a morphism $f'$ factor through $\I$ such that $f=ha_1+f'$ in $\C$. Then from the discussion above, $f'$ and thus $f$ factor through $(a_1~a_2)$. It follows that $(a_1~ a_2)$ is a left $add(ind(\T\oplus\I) \setminus \{T\})$-approximation in $\C$. Therefore the right exchange triangle of $T$ in $\T\oplus\I$ is a direct summand of the triangle $\bigstar$ as a complex. Note that each indecomposable direct summand of $E$ is not in $add I_2$, then it belongs to the middle term of the right exchange triangle of $T$ in $\T\oplus\I$. Therefore the right exchange triangle of $T$ is of the form of triangle $\divideontimes$. Similarly we can prove the case of the left exchange triangle of $T$ in $\T\oplus\I$.
\item The right exchange triangle ${T}\s{(a_1~ a_2)}\longrightarrow E\oplus I\longrightarrow {T'}\longrightarrow {T[1]}$ in $\C$ induces an triangle ${T}\s{\underline{a_1}}\longrightarrow E\longrightarrow {T'}\longrightarrow {T\langle1\rangle}$ in $\C'$\cite{IY08}. It is easy to see that $\underline{a_1}$ is a minimal left $add(ind\T \setminus \{T\})$-approximation in $\C'$ due to $(a_1,a_2)$ is a minimal left $add(ind(\T\oplus\I) \setminus \{T\})$-approximation in $\C$. Therefore by the uniqueness of the minimal left approximation, ${T}\s{\underline{a_1}}\longrightarrow E\longrightarrow {T'}\longrightarrow {T\langle1\rangle}$ is the right exchange triangle of $T$ in $\T$. Similarly we can prove the case of the left exchange triangle.
\end{enumerate}
\end{proof}

\begin{prop}\label{thm of CSS in SFTC}
The subfactor category $^{\bot}\I[1]/\I$ has a cluster structure.
\end{prop}
\begin{proof}
Let $\T$ be a cluster tilting subcategory of $\C'$,
then it follows from the above lemma that the quiver $Q(\T)$ is a full subquiver of $Q(\T\oplus\I)$ by deleting vertices given by isomorphism classes of indecomposable objects in $\I$. Thus $Q(\T)$ has no loops nor $2$-cycles since $\C$ has a cluster structure. Therefore $\C'$ has a cluster structure.
\end{proof}

\subsection{Cluster substructures in cotorsion pairs}\label{Sec TriCot}

In \cite{ZZ12}, the authors studied the cotorsion pairs in a $2$-Calabi-Yau triangulated category with cluster tilting objects and give a classification of cotorsion pairs in the category. For a 2-Calabi-Yau triangulated category $\C$ with a cluster tilting subcategory, we have the following analogy of Proposition 5.3 in \cite{ZZ12}. Before state the proposition, we recall a result in \cite{BIRS09}, which is a key point in the following study of cluster substructures in cotorsion pairs.

\begin{lem}(Proposition II.2.3 \cite{BIRS09})\label{lem of decop of SFTC}
Let  $(\X,\Y)$  be a cotorsion pair of $\C$. Then the core $\I=\X\cap \Y$ is a functorially finite rigid subcategory and there is a decomposition of triangulated category $^{\bot}\I[1]/\I=\X/\I\oplus \Y/\I$.
\end{lem}

\begin{prop}\label{Prop of x-cts}
Let  $(\X,\Y)$  be a cotorsion pair of $\C$ with core $\I=\X\cap \Y$. Assume that there is a cluster tilting subcategory $\T$ contains $\I$ as a subcategory (this is true if $\C$ has a cluster tilting object\cite{IY08,DK08,ZZ}). Then
\begin{enumerate}
\item Any cluster tilting subcategory $\T'$ containing $\I$ as a subcategory can be uniquely written as $\T'=\T'_{\X}\oplus \I \oplus \T'_{\Y} $, such that $\T'_{\X}\oplus \I$ is a $\X$-cluster tilting subcategory and  $\T'_{\Y} \oplus \I$ is a $\Y$-cluster tilting subcategory.
\item Any $\X$-cluster tilting subcategory is of the form $\T'_{\X}\oplus \I$, where $\T'=\T'_{\X}\oplus \I \oplus \T'_{\Y} $ is a cluster tilting subcategory of $\C$ and $\T'_{\Y}\oplus \I$ is a $\Y$-cluster tilting subcategory.
\item The correspondence $\T' \rightarrow \T'_\X \oplus \I \oplus \T'_\Y$ gives a bijection between the cluster tilting subcategories in $\C$ containing $\I$ as a subcategory and the pairs of the $\X$-cluster tilting subcategories and the $\Y$-cluster tilting subcategories.
\end{enumerate}
\end{prop}

\begin{proof}
Because $\I\subseteq \T$ and $\T$ is rigid, we have $\T\subseteq {^{\bot}\I[1]}$; moreover due to the decomposition $^{\bot}\I[1]/\I=\X/\I\oplus \Y/\I$, $\T$ can be decomposed as $\T_\X\oplus \T_\Y$ in the quotient category $^{\bot}\I[1]/\I$, where $\T_{\X}\subseteq \X$ and $\T_{\Y}\subseteq \Y$. Thus in $\C$, $\T=\T_{\X}\oplus \I \oplus \T_{\Y}$. By using the decomposition $^{\bot}\I[1]/\I=\X/\I\oplus \Y/\I$, the proof of (1) is similar to the proof of Proposition 5.3(1) in \cite{ZZ12}. For the statement (2), note that the existence of the cluster tilting subcategory $\T$ guarantees that any $\X$-cluster tilting subcategory can be extended as a cluster tilting subcategory in $\C$. In fact, it is clear that any $\X$-cluster tilting subcategory contains $\I$ as a subcategory, and let $\T'_{\X}\oplus \I$ be such a $\X$-cluster tilting subcategory. Then we claim that $\T'_{\X}$ is a cluster tilting subcategory in $\X/\I$. Let $X$ be an object in $\X/\I$, because $\T'_\X\oplus \I$ is contravariantly finite in $\X$, there exist an object $T\in \T'_\X\oplus \I$ and a morphism $f\in \Hom_\C(T,X)$ such that $f\cdot :\Hom_{\C}(-,T)\rightarrow \Hom_{\C}(-,X)$ is surjective on $\T'_\X\oplus \I$. Then it is easy to check that $\underline{f}\cdot :\Hom_{\X/\I}(-,T)\rightarrow \Hom_{\X/\I}(-,X)$ is surjective on $\T'_\X$. Therefore $\T'_\X$ is contravariantly finite in $\X/\I$. Similarly, $\T'_\X$ is convariantly finite and thus functorially finite in $\X/\I$.
Let $X$ be an object in $\X/\I$ such that ${\Ext^1}_{\X/\I}(T,X)=0$ for any $T \in \T'_\X$. We now prove that $X\in \T'_\X$. In fact, from ${\Ext^1}_{{^{\bot}\I[1]}/\I}(T,X)={\Ext^1}_{\X/\I}(T,X)=0$, we have ${\Ext^1}_{\C}(T,X)=0$ due to Lemma 4.8 in \cite{IY08}. Then we have ${\Ext^1}_{\C}(T,X)=0$ for any $T \in \T'_\X\oplus \I$ since ${\Ext^1}_{\C}(I,X)=0$ for any $I\in \I$. Then $X \in \T'_\X\oplus \I$ since $\T'_\X\oplus \I$ is a $\X$-cluster tilting subcategory in $\X$. Therefore $X$ belongs to $\T'_\X$. We have proved the claim. Note that $\T_\Y\oplus \I$ is a $\Y$-cluster tilting subcategory in $\Y$ from the statement (1), then by a similar argument, $\T_\Y$ is a cluster tilting subcategory in $\Y/\I$. Then from Lemma \ref{lem of decop of SFTC}, it is clear that $\T'_{\X}\oplus \T_{\Y}$ is a cluster tilting subcategory in $^{\bot}\I[1]/\I$. Therefore $\T'_{\X}\oplus \I \oplus \T_{\Y}$ is a cluster tilting subcategory in $\C$ (Theorem 4.9 in \cite{IY08}). Finally the statement (3) follows from (1) and (2).
\end{proof}

\begin{defn}
For a $\X$-cluster tilting subcategory $\T'_{\X}\oplus \I$ in $\X$, we define $Q(\T'_{\X}\oplus \I)$ as an ice quiver with the exchangeable vertices given by the isomorphism classes of indecomposable objects in $\T'_\X$ and the frozen vertices given by the isomorphism classes of indecomposable objects in $\I$. For two vertices  $T_i$ and $T_j$ (not both the frozen vertices), the number of arrows from $T_i$ to $T_j$ is given by the dimension of irreducible morphism space $irr(T_i,T_j)$ in $\T'_{\X}\oplus \I$.
\end{defn}

Now we state the main result in this subsection.

\begin{thm}\label{Them in TriC}
Let $\C$ be a 2-Calabi-Yau triangulated category with a cluster tilting subcategory and $(\X,\Y)$  be a cotorsion pair of $\C$ with core $\I=\X\cap \Y$. Assume that there is a cluster tilting subcategory in $\C$ contains $\I$ as a subcategory. If the cluster tilting subcategories in $\C$ forms a cluster structure, then the $\X$-cluster tilting subcategories form a cluster structure of $\X$ with coefficient subcategory $\I$: it is a cluster substructure of $\C$ in the sense of section II.2\cite{BIRS09}, more precisely, the following conditions are satisfied:
 \begin{enumerate}
\item For each $\X$-cluster tilting subcategory  $\T_{\X}\oplus \I$ in $\X$ and an indecomposable object $T_0$ in $\T_{\X}$, there is  a unique (up to isomorphism) indecomposable object $T'_0\ncong T_0$ in $\X$ such that $\T'_\X\oplus \I:= add(ind\T_\X \setminus \{T_0\} \cup \{T'_0\})\oplus \I$ is a $\X$-cluster tilting subcategory in $\X$.
\item In the situation of (1), there are triangles
$${T_0}\s{f}\longrightarrow E\oplus I\s{g}\longrightarrow {T'_0}\s{}\longrightarrow {T_0[1]}~~~~~~~~~~~~ and$$
$${T'_0}\s{s}\longrightarrow E'\oplus I'\s{t}\longrightarrow {T_0}\s{}\longrightarrow {T'_0[1]}~~~~~~~~~~~~~~~~~~$$
in $\X$, where $g$ and $t$ are minimal right $(\T_\X \cap \T'_\X)\oplus \I $-approximation and $f$ and $s$ are minimal left $(\T_\X \cap \T'_\X)\oplus \I $-approximation. The subcategory $\mu_{T_0}(\T_{\X}\oplus \I):=\T'_\X\oplus \I$ is called the mutation of $\T_{\X}\oplus \I$ at $T_0$. These two triangles are called the right exchange triangle and the left exchange triangle of $T_0$ in $\T_\X\oplus\I$ respectively.
\item For each $\X$-cluster tilting subcategory  $\T_{\X}\oplus \I$ in $\X$, there are no loops nor $2$-cycles in the ice quiver $Q({\T_\X}\oplus \I)$.
\item In the situation of (1), passing from $Q({\T_\X}\oplus \I)$ to $Q({\T'_\X}\oplus \I)$ is given by the Fomin-Zelevinsky mutation at the vertex of $Q({\T_\X}\oplus \I)$ corresponding to $T_0$.
\item There is a subcategory $\B$ of $\C$ such that  $\mu_{T_n}\circ \circ \circ \mu_{T_0}(\T_{\X}\oplus \I)\oplus \B$ is a cluster tilting subcategory in $\C$ for any finite
sequence of mutations $\mu_{T_n}\circ \circ \circ \mu_{T_0}(\T_{\X}\oplus \I)$ of $\T_{\X}\oplus \I$.
\end{enumerate}
\end{thm}

\begin{proof}
\begin{enumerate}
\item From Proposition\ref{Prop of x-cts}, we can assume that each $\X$-cluster tilting subcategory is of the form $\T_{\X}\oplus \I$, where $\T=\T_{\X}\oplus \I \oplus \T_{\Y} $ is a cluster tilting subcategory in $\C$ with $\T_{\Y} \in \Y$.
Let
\begin{equation}\label{Ex1}
~~~~~~~~~~~~{T_0}\s{f}\longrightarrow E\oplus I\s{g}\longrightarrow {T'_0}\s{}\longrightarrow {T_0[1]}~~~~~~~and
\end{equation}
\begin{equation}\label{Ex2}
{T'_0}\s{s}\longrightarrow E'\oplus I'\s{t}\longrightarrow {T_0}\s{}\longrightarrow {T'_0[1]}
\end{equation}
be the exchange triangles of $\T$ in $\C$ with $I\in \I$ and $E$ and $E'$ have no direct summands in $\I$.
Then from Lemma \ref{lem in 3.2}, they induce exchange triangles
\begin{equation}\label{Ex3}
~~~~~~~~~~~~{T_0}\s{\underline{f}}\longrightarrow E\s{\underline{g}}\longrightarrow {T'_0}\s{}\longrightarrow {T_0 \langle 1 \rangle}~~~~~~~and
\end{equation}
\begin{equation}\label{Ex4}
{T'_0}\s{\underline{s}}\longrightarrow E'\s{\underline{t}}\longrightarrow {T_0}\s{}\longrightarrow {T'_0 \langle 1 \rangle}
\end{equation}
in the subfactor category $^{\bot}\I[1]/\I$ respectively. Moreover we have $E \in {\T_\X}\oplus \I$ in the triangle (\ref{Ex1}), this is because $^{\bot}\I[1]/\I=\X/\I\oplus \Y/\I$ and thus any morphism from $T_0$ to $\T_\Y$ factor through $\I$. Similarly, $E' \in {\T_\X}\oplus \I$ in the triangle (\ref{Ex2}). Since $\X/\I$ is a triangulated subcategory of $^{\bot}\I[1]/\I$, the triangles $(\ref{Ex3})$ and $(\ref{Ex4})$ are both in $\X/\I$ and thus $ T'_0$ belongs to $\X$.
Because $\T'_\X\oplus \I \oplus \T_{\Y}$ is a cluster tilting subcategory in $\C$, $\T'_\X\oplus \I$ is a $\X$-cluster tilting subcategory due to Proposition\ref{Prop of x-cts}. The uniqueness of $\T'_\X\oplus \I$ comes from the uniqueness of $\T'_\X\oplus \I \oplus \T_{\Y}$. We have proved the statement.
\item It is only need to show that in the triangles $(\ref{Ex1})$ and $(\ref{Ex2})$, $g$ and $t$ are minimal right $(\T_\X \cap \T'_\X)\oplus \I $-approximation and $f$ and $s$ are minimal left $(\T_\X \cap \T'_\X)\oplus \I $-approximation. This easily follows from the fact that $(\ref{Ex1})$ and $(\ref{Ex2})$ are the exchange triangles.
\item  We claim that the arrows between any two vertices (not both frozen) in the ice quiver $Q(\T_\X\oplus \I)$ are coincide with the arrows between these vertices in the quiver $Q(\T)$.
In fact, given vertices $T_0$ and $T_1$ in the ice quiver $Q(\T_\X\oplus \I)$ with $T_0$ exchangeable, by unifying the vertices in the quiver and the isomorphism classes of indecomposable objects in the subcategory, the proof of the statement (2) shows that $\T_\X\oplus\I$ and $\T$ have the same exchange triangles at $T_0$, thus the number of arrows from $T_0$ to $T_1$ and the number of arrows from $T_1$ to $T_0$ in the quivers $Q(\T_\X\oplus \I)$ and $Q(\T)$ are all determined by the degree of $T_1$ in $E\oplus I$ and $E'\oplus I'$ respectively.
Thus the conclusion follows from the assumption that $\C$ has a cluster structure.
\item  We can prove this similar to the proof of Theorem II.1.6 in \cite{BIRS09}.
\item Let $\B=\T_\Y$ be the subcategory of $\Y$ in the proof of statement (1). Note that any finite mutation $\mu_{T_n}\circ \circ \circ \mu_{T_0}(\T_{\X}\oplus \I)$ is a cluster tilting subcategory in $\X/\I$, then from the proof of Proposition\ref{Prop of x-cts}, $\mu_{T_n}\circ \circ \circ \mu_{T_0}(\T_{\X}\oplus \I)\oplus \T_\Y$ is a cluster tilting subcategory in $\C$.
\end{enumerate}
\end{proof}


\subsection{Cluster structures and rooted cluster algebras}\label{Sec TriCC}

In this subsection, we fix the following settings. We always assume that $\C$ is a $2$-Calabi-Yau triangulated category with a cluster structure given by its cluster tilting subcategories. Let $(\X,\Y)$ be a cotorsion pair of $\C$ with core $\I=\X\cap \Y$, where we assume that $\I$ is a subcategory of a cluster tilting subcategory $\T$ in $\C$. Then Proposition \ref{Prop of x-cts} guarantees that we can write $\T=\T_{\X}\oplus \I \oplus \T_{\Y} $ with $\T_{\X}\oplus \I$ being $\X$-cluster tilting and $\T_{\Y}\oplus \I$ being $\Y$-cluster tilting. We collect the following notations which we use in this subsection.
\begin{itemize}
\item $Q(\T):$ the quiver of the cluster tilting subcategory $\T$.
\item $Q(\T_\X\oplus \I):$ the ice quiver of the $\X$-cluster tilting subcategory $\T_\X\oplus \I$.
\item $Q(\T_\Y\oplus \I):$ the ice quiver of the $\Y$-cluster tilting subcategory $\T_\Y\oplus \I$.
\item $Q(\T_{\I}):$ the ice quiver by freezing the vertices in $Q(\T)$ which are determined by the isomorphism classes of the indecomposable objects in $\I$.
\item $Q(\T\setminus \I):$ the quiver of the cluster tilting subcategory $\T\setminus \I$ in the subfactor triangulated category $^{\bot}\I[1]/\I$.
\end{itemize}
The rooted cluster algebras corresponding to the above quivers are denoted by $\A(\T)$, $\A(\T_{\X}\oplus\I), \A(\T_{\Y}\oplus\I), \A(\T_\I)$ and $\A(\T\setminus \I)$ respectively.
\begin{itemize}
\item $R(\T):$ the subcategory of $\C$ additive generated by rigid objects which are reachable from $\T$ by a finite number of mutations, where the mutation is defined in Subsection \ref{Sec TriP}.
\item $R(\T_\X\oplus \I):$ the subcategory of $\C$ additive generated by rigid objects which are reachable from $\T_\X\oplus \I$ by a finite number of mutations, where the mutation is defined in Theorem\ref{Them in TriC}(2).
\item $R(\T_\Y\oplus \I):$ the subcategory of $\C$ additive generated by rigid objects which are reachable from $\T_\Y\oplus \I$ by a finite number of mutations.
\item $R(\T_{\I}):$ the subcategory of $\C$ additive generated by rigid objects which are reachable from $\T$ by a finite number of mutations, where the mutations are not at the indecomposable objects in the subcategory $\I$.
\item $R(\T\setminus \I):$ the subcategory of $\C$ additive generated by rigid objects which are reachable from $\T\setminus \I$ by a finite number of mutations in the subfactor triangulated category $^{\bot}\I[1]/\I$, where we view these reachable rigid objects as objects in $\C$.
\end{itemize}
Then from Lemma \ref{lem in 3.2}, it is not hard to see that $indR(\T_{\I})=indR(\T\setminus \I)\sqcup ind\I$. We have a sequence $R(\T_\X\oplus \I)\subseteq R(\T_{\I})\subseteq R(\T)$ of inclusions, where the first one follows from Theorem \ref{Them in TriC}(2) and the second one is clear. Denote by $i_1: R(\T_\X\oplus \I)\rightarrow R(\T_{\I})$ and $i_2: R(\T_\I)\rightarrow R(\T)$ the natural embedding functors under the above inclusions. We define a canonical functor $p: R(\T_\I) \rightarrow R(\T\setminus \I)$, which is an identity on $R(\T\setminus \I)$ and maps an object in $\I$ to zero object.\\

The cluster map, which is called the cluster character in \cite{Palu08}, is defined in \cite{BIRS09}. One can use it to transform a cluster structure in $\C$ to a cluster algebra. We recall the following definition of cluster map from \cite{BIRS09} and \cite{FK10} in our settings. For more details, we refer to \cite{BIRS09} and \cite{FK10}.

\begin{defn}
Denote by $\Q(X_\T)$ the rational function field of $X_\T$, where $X_\T$ is the indeterminate set which is indexed by the isomorphism classes of indecomposable objects in $\T$. A map $\varphi$ from $\R(\T)$ to $\Q(X_\T)$ is called a cluster map if the following conditions are satisfied:
\begin{enumerate}
\item For the object $M \cong M'$, we have $\varphi(M)=\varphi(M')$.
\item For any indecomposable object $T_i$ in $\T$, we have $\varphi(T_i)=x_i$ where $x_i \in X_\T$ is the element indexed by $T_i$.
\item For any $M$ and $N$ in $\R(\T)$ with \textrm{dim}$\Ext^1_{\C}(M,N)=1$ (thus \textrm{dim}$\Ext^1_{\C}(N,M)=1$), we have $\varphi(M)\varphi(N)=\varphi(V)+\varphi(V')$ where $V$ and $V'$ are in the non-split triangles
$${M}\s{}\longrightarrow V\longrightarrow N\longrightarrow {M[1]}~~~~~~~~~~~~~~~~~~and$$
$${N}\s{}\longrightarrow V'\longrightarrow M\longrightarrow {N[1].}~~~~~~~~~~~~~~~~~~~~~~~$$
\item For any $M$ and $N$ in $\R(\T)$, we have $\varphi(M\oplus N)=\varphi(M)\varphi(N)$. In particular, $\varphi(0)=1$.
\end{enumerate}
\end{defn}

Then the cluster map $\varphi$ constructs a connection between a cluster structure (cluster substructure respectively) of $\C$ and the rooted cluster algebras (the rooted cluster subalgebras respectively).
More precisely, we have the following proposition, where the first statement can be proved similar to the proof of Proposition 2.3 in \cite{FK10} and the second one can be easily derived from the first one and Theorem\ref{Them in TriC}. The last one clearly follows from the first one.

\begin{prop}\label{prop of RCMap}
\begin{enumerate}
\item The map $\varphi$ induces a surjection from the set of isomorphism classes of indecomposable objects in $\R(\T)$ onto the set of cluster variables in $\A(\T)$, and also induces a surjection from the set of cluster tilting subcategories reachable from $\T$ onto the set of clusters of $\A(\T)$.
\item The map $\varphi$ induces a surjection $\varphi_1$ from the set of isomorphism classes of indecomposable objects in $\R(\T_\X\oplus\I)$ onto the set of cluster variables in $\A(\T_{\X}\oplus\I)$, and also induces a surjection from the set of $\X$-cluster tilting subcategories reachable from $\T_\X\oplus\I$ onto the set of clusters of $\A(\T_{\X}\oplus\I)$.
\item The map $\varphi$ induces a surjection $\varphi_2$ from the set of isomorphism classes of indecomposable objects in $\R(\T_\I)$ onto the set of cluster variables in $\A(\T_\I)$, and also induces a surjection from the set of cluster tilting subcategories reachable from $\T$ by finite number of mutations not at indecomposable objects in $\I$ onto the set of clusters of $\A(\T_\I)$.
\end{enumerate}
\end{prop}

From Theorem \ref{thm of CSS in SFTC}, the subfactor category $^{\bot}\I[1]/\I$ inherits a cluster structure. By viewing $R(\T\setminus \I)$ as a subcategory in $^{\bot}\I[1]/\I$, we denote by $\varphi'$ the cluster map from $R(\T\setminus \I)$ to $\A(\T\setminus \I)$.
Now we state our main result in this section.
\begin{thm}\label{thm of Last}
Under the above settings,
\begin{enumerate}
\item The rooted cluster algebras $\A(\T_{\X}\oplus\I)$ and $\A(\T_{\Y}\oplus\I)$ are both rooted cluster subalgebras of $\A(\T_\I)$ and thus rooted cluster subalgebras of $\A(\T)$. Moreover, $\A(\T_\I)$ is the glue of $\A(\T_{\X}\oplus\I)$ and $\A(\T_{\Y}\oplus\I)$ at $\varphi(\I)$.
\item Any rooted cluster subalgebra of $\A(\T)$ with coefficient set $\varphi(\I)$ such that $\I$ is functorially finite in $\C$ is of the form $\A(\T_{\X'}\oplus\I)$, where $\T_{\X'}\oplus \I$ is $\X'$-cluster tilting in a cotorsion pair $(\X',\Y')$ with core $\I$.
\item The correspondence $(\X',\Y')\mapsto (\A(\T_{\X'}\oplus\I),\A(\T_{\Y'}\oplus\I))$ gives the following bijection:
\begin{center}
\{cotorsion pairs in $\C$ with core $\I$\}\\
$\Updownarrow$\\
\{complete pairs of rooted cluster subalgebras of $\A(\T)$ with coefficient set $\varphi(\I)$ such that $\I$ is functorially finite in $\C$\}.
\end{center}
This bijection induces the following bijection:
\begin{center}
\{t-structures in $\C$\}\\
$\Updownarrow$\\
\{complete pairs of rooted cluster subalgebras of $\A(\T)$ without coefficients\}.
\end{center}
\item The specialization at $\varphi(\I)$ induces a rooted cluster surjection $\pi$ from $\A(\T_\I)$ to $\A(\T\setminus\I)$.
\item We have the following commutative diagram:
$$\xymatrix{&\R(\T\setminus\I)\ar[dd]^{\varphi'}&&&\\
&&\R(\T_\I)\ar[ul]_{p}\ar[dd]^{\varphi_2}\ar[rr]^{i_2}&&\R(\T)\ar[dd]^{\varphi}\\
\R(\T_\X\oplus\I)\ar[urr]^{i_1}\ar[dd]^{\varphi_1}&\A(\T\setminus\I)&&&\\
&&\A(\T_\I)\ar[ul]_{\pi}\ar[rr]^{j_2}&&\A(\T)\\
\A(\T_{\X}\oplus\I)\ar[urr]^{j_1}&&&&
}$$
where $j_1$ and $j_2$ are injections arising from subalgebras in the first statement.
\end{enumerate}
\end{thm}

\begin{proof}
\begin{enumerate}
\item On the one hand, from Theorem \ref{Them in TriC}, the quiver $\Q(\T_{\X}\oplus \I)$ is a full subquiver of $Q(\T_\I)$. On the other hand, from the decomposition of the triangulated category $^{\bot}\I[1]/\I=\X/\I\oplus \Y/\I$, the morphisms between objects in $\T_\X$ and $\T_\Y$ factor through $\I$ and thus in the quiver $Q(\T_\I)$ there are no arrows between the vertices in $ind\T_{\X}$ and the vertices in $ind\T_{\Y}$. Therefore $Q(\T_{\X}\oplus \I)$ is a glue of some connected components of $Q(\T_\I)$. Thus from Theorem \ref{Prop of inj}, $\A(\T_{\X}\oplus\I)$ is a rooted cluster subalgebras of $\A(\T_\I)$.
Similarly, $\A(\T_{\Y}\oplus\I)$ is a rooted cluster subalgebras of $\A(\T_\I)$. Moreover, it is clear that $\A(\T_\I)$ is a glue of $\A(\T_{\X}\oplus\I)$ and $\A(\T_{\Y}\oplus\I)$ along $\varphi(\I)$. Since $\A(\T_\I)$ is a rooted cluster subalgebra of $\A(\T)$ by Definition-Proposition\ref{inj of frozen}, $\A(\T_{\X}\oplus\I)$ and $\A(\T_{\Y}\oplus\I)$ are rooted cluster subalgebras of $\A(\T)$.
\item From Theorem \ref{Prop of inj}, any rooted cluster subalgebra of $\A(\T)$ with coefficient set $\varphi(\I)$ is of the form $\A(\T'\oplus \I)$ where $\T'\oplus \I \subseteq \T$, and in the quiver $Q(\T)$, there are no arrows between vertices in $ind\T'$ and vertices in $ind\T \setminus ind{(\T'\oplus \I)}$. Let $\T''=add(ind\T \setminus ind{(\T'\oplus \I)})$ be a subcategory of $\C$. Now we consider the pair $(\T',\T'')$ in the subfactor category $^{\bot}\I[1]/\I$, which is a $2$-Calabi-Yau triangulated category since $\I$ is functorially finite in $\C$. Note that $\T'\oplus\T''$ is a cluster tilting subcategory in $^{\bot}\I[1]/\I$ and $\Hom_{^{\bot}\I[1]/\I}(\T',\T'')=\Hom_{^{\bot}\I[1]/\I}(\T'',\T')=0$. Thus we have $^{\bot}\I[1]/\I=(\T'\oplus\T'')\ast (\T'\langle1\rangle\oplus\T''\langle1\rangle)=\T'\ast \T'\langle1\rangle \oplus \T''\ast \T''\langle1\rangle=\C_1\oplus \C_2$ as a decomposition of triangulated category by Proposition 3.5\cite{ZZ12}. Let $\pi : ^{\bot}\I[1]\rightarrow ^{\bot}\I[1]/\I$ be the natural projection. Then because $(\C_1,\C_2)$ is a cotorsion pair in $^{\bot}\I[1]/\I$ with core $\{0\}$, $(\X',\Y')=(\pi^{-1}(\C_1),\pi^{-1}(\C_2))$ is a cotorsion pair in $\C$ with core $\I$ by Theorem 3.5\cite{ZZ11}. It is clear that $\T' \oplus \I$ is $\X'$-cluster tilting.
\item The first assertion follows from above statements (1) and (2). The second assertion follows from the fact that $(\X,\Y))$ is a t-structure if and only if $\I=\{0\}$ (Proposition 2.9 \cite{ZZ11}).
\item It is easily follows from Lemma \ref{lem in 3.2} that $Q(\T\setminus\I)$ is a full subquiver of $Q(\T_\I)$ by deleting all the frozen vertices. Therefore the conclusion follows from Proposition \ref{prop of spec for frozen}.
\item Note that the maps $\varphi_1, \varphi_2$ and $\varphi'$ are all induced by $\varphi$, and the injections $i_1,i_2,j_1$ and $j_2$, the surjection $p$ and $\pi$ are all canonical, thus the commutative diagram is natural valid.
\end{enumerate}
\end{proof}

\begin{rem}
\begin{enumerate}
\item If $\C$ has a cluster tilting object, then any rigid subcategory $\I$ is additive generated by an object in $\C$\cite{DK08,ZZ}. Thus $\I$ is functorially finite in $\C$. Therefore there is a bijection between the following two sets:
\begin{center}
\{cotorsion pairs in $\C$ with core $\I$\}\\
$\Updownarrow$\\
\{complete pairs of rooted cluster subalgebras of $\A(\T)$ with coefficient set $\varphi(\I)$\}.
\end{center}
In this case, each rooted cluster subalgebra $\A(\T_{\X}\oplus\I)$ of $\A(\T)$ has a 2-Calabi-Yau categorification by the stably 2-Calabi-Yau category $\X$ in the sense of \cite{BIRS09,FK10}.
\item It follows from above statement that if a rooted cluster algebra $\A(\S)$ has a 2-Calabi-Yau categorification by a 2-Calabi-Yau triangulated category with a cluster tilting object, then any rooted cluster subalgebra of $\A(\S)$ has a 2-Calabi-Yau categorification by a stably 2-Calabi-Yau category.
\end{enumerate}
\end{rem}

\begin{cor}\label{Final cor}
Under the settings assumed at the beginning of this subsection, we unify the following six kinds of decompositions:
\begin{enumerate}
\item The decomposition of the triangulated category $^{\bot}\I[1]/\I$ in the sense of Proposition II.2.3 in \cite{BIRS09}.
\item The decomposition of the cluster tilting subcategory $\T$ in $^{\bot}\I[1]/\I$ in the sense of Definition 3.3 in \cite{ZZ12}.
\item The decomposition of the ice quiver $Q(\T_\I)$ in the sense of Definition-Proposition \ref{prop of decop of IVQ}.
\item The decomposition of the exchange matrix $B(\T_\I)$ in the sense of Definition-Proposition \ref{prop of decop of Matrix}.
\item The decomposition of the seed $\S(\T_\I)$ in the sense of Definition \ref{def of indec of seed}.
\item The decomposition of the rooted cluster algebra $\A(\T_\I)$ in the sense of Theorem \ref{inj of indec comp}.
\end{enumerate}
\end{cor}
\begin{proof}
It is proved in Theorem 3.10 \cite{ZZ12} that the first two decompositions are unified. We have proved in subsections \ref{Sec RCMIce} and \ref{Sec RCMT} that the last four decompositions are unified. It can be directly derived from definitions that the decompositions of $\T$ and $Q(\T_\I)$ are unified.
\end{proof}
It follows from Theorem \ref{thm of Last} and Corollary \ref{Final cor} that we can classify the cotorsion pairs in $\C$ with core $\I$ by gluing indecomposable components of $\A(\T_\I)$ (or equivalently of $Q(\T_\I)$ , of $B(\T_\I)$ and of $\S(\T_\I)$ respectively). In fact, each way of gluing all the indecomposable components of $\A(\T_\I)$ to a complete pair of rooted cluster subalgebras of $\A(\T_\I)$ with coefficient set $\varphi(\I)$ gives a unique cotorsion pair in $\C$ with core $\I$.

\section*{Acknowledgements}

The authors wish to thank Jie Zhang and Wuzhong Yang for helpful discussions on the topic. After finishing the paper, we know from arXiv that Sira Gratz has also some similar results in Section 2.2 on ideal rooted cluster morphisms in \cite{G14}. Both authors would like to thank the anonymous reviewer for careful reading and valuable suggestions.

\end{document}